\documentclass[a4paper,11pt]{amsart}

\usepackage{amsmath}
\usepackage{amscd}
\usepackage{amsxtra}
\usepackage{amssymb}
\usepackage{amsthm}

\newtheorem{thm}{Theorem}[section]
\newtheorem{cor}[thm]{Corollary}
\newtheorem{lem}[thm]{Lemma}
\newtheorem{prop}[thm]{Proposition}
\newtheorem{conj}[thm]{Conjecture}

\newtheorem*{thm*}{Theorem}
\newtheorem*{prop*}{Proposition}
\newtheorem*{cor*}{Corollary}
\newtheorem*{conj*}{Conjecture}

\theoremstyle{definition}
\newtheorem{defn}[thm]{Definition}

\theoremstyle{remark}
\newtheorem{ex}[thm]{Example}
\newtheorem{rem}[thm]{Remark}

\newtheorem{conv}[thm]{Convention}

\newcommand{\ka}{{\mathcal A}}
\newcommand{\kb}{{\mathcal B}}
\newcommand{\kc}{{\mathcal C}}
\newcommand{\kd}{{\mathcal D}}
\newcommand{\ke}{{\mathcal E}}
\newcommand{\kf}{{\mathcal F}}

\newcommand{\kl}{{\mathcal L}}
\newcommand{\km}{{\mathcal M}}
\newcommand{\ko}{{\mathcal O}}
\newcommand{\kp}{{\mathcal P}}

\newcommand{\kt}{{\mathcal T}}

\newcommand{\IC}{{\mathbb C}}

\newcommand{\IH}{{\mathbb H}}

\newcommand{\IP}{{\mathbb P}}

\newcommand{\IR}{{\mathbb R}}

\newcommand{\IZ}{{\mathbb Z}}

\newcommand{\gu}{\mathfrak{U}}

\newcommand{\fa}{{\mathfrak a}}
\newcommand{\fg}{{\mathfrak g}}

\newcommand{\fk}{{\mathfrak k}}
\newcommand{\fm}{{\mathfrak m}}

\newcommand{\curly}[1]{\mathcal{#1}}

\newcommand{\DR}{\mathrm{\mathbb R}} 
\newcommand{\DL}{\mathrm{\mathbb L}}

\newcommand{\vphi}{\varphi}

\newcommand{\eps}{\varepsilon}

\newcommand{\id}{{\rm id}}
\newcommand{\tensor}{\otimes}

\newcommand{\ra}{\rightarrow}
\newcommand{\lra}{\longrightarrow}

\newcommand{\isom}{\cong}

\newcommand{\del}{\partial}
\newcommand{\union}{\cup}
\newcommand{\dotcup}{\ensuremath{\mathaccent\cdot\cup}}
\newcommand{\dunion}{\dotcup}

\newcommand{\xlra}[1]{\overset{#1}{\lra}}

\newcommand{\sra}{\xlra{\sim}}

\newcommand{\qtext}[1]{\quad\text{#1}\quad}
\newcommand{\stext}[1]{\;\text{#1}\;}

\newcommand{\Set}[2]{\left\{\, #1 \;|\; #2 \,\right\}}

\DeclareMathOperator{\td}{td}
\DeclareMathOperator{\pr}{pr}

\newcommand{\HH}{\mathrm{H}}

\newcommand{\half}{\frac{1}{2}}
\newcommand{\<}{\langle}
\renewcommand{\>}{\rangle}

\newcommand{\vsum}{\oplus}

\newcommand{\inj}{\hookrightarrow}
\newcommand{\surj}{\twoheadrightarrow}

\renewcommand{\HH}{H}

\newcommand{\dual}{\hspace{0cm}^\vee}

\renewcommand{\forall}{\text{ for all }}

\newcommand{\GLT}{\widetilde{Gl}_2^+(\IR)}
\newcommand{\GL}{Gl_2^+(\IR)}
\newcommand{\WT}{\widetilde{W}}
\newcommand{\HT}{\tilde{\HH}}
\newcommand{\NDel}{\Delta^{>0}}

\newcommand{\limt}{\underset{t \ra \infty}{\lim}}
\newcommand{\KM}{\overline{KM}}

\usepackage[all,cmtip]{xy}
\usepackage{graphicx}
\usepackage{lmodern}
\usepackage[latin1]{inputenc}
\usepackage{verbatim}

\newcommand{\pd}{\mathfrak{D}} 
\newcommand{\cd}{\mathcal{D}} 

\title[Cusps and stability conditions]{Cusps of the K\"ahler moduli space and stability conditions on K3 surfaces}
\author{Heinrich Hartmann}
\date{\today}

\newcommand{\settocdepth}[1]{
\addtocontents{toc}{\protect\setcounter{tocdepth}{#1}}}

\begin{document}

\begin{abstract}
  In \cite{Ma1} S. Ma established a bijection between Fourier--Mukai partners 
  of a K3 surface and cusps of the K\"ahler moduli space. 
  The K\"ahler moduli space can be described as a quotient of Bridgeland's
  stability manifold. We study the relation between stability
  conditions $\sigma$ near to a cusp and the associated Fourier--Mukai
  partner $Y$ in the following ways.
  (1) We compare the heart of $\sigma$ to the 
  heart of coherent sheaves on $Y$. (2) We construct $Y$ as moduli
  space of $\sigma$-stable objects.
    
  An appendix is devoted to the group of auto-equivalences of $\kd^b(X)$
  which respect the component $Stab^\dagger(X)$ of the stability
  manifold.
\end{abstract}

\maketitle

\settocdepth{1}

\tableofcontents 
\hspace{1cm}

\section{Introduction}

\label{sec:Introduction}

Let $X$ be a projective K3 surface over the complex numbers, and let
$\kt = \cd^b(X)$ be the bounded derived category of coherent sheaves
on $X$.

We associate to $\kt$ the complexified K\"ahler moduli space $KM(\kt)$ by the following
procedure. Let $N(\kt)$ be the numerical Grothendieck group of $\kt$ endowed with
the (negative) Euler pairing. We consider the following period domain 
\[ \pd(\kt)=\{ [z] \in \IP(N(\kt)_{\IC}) \,|\, z.z = 0, z.\bar{z}>0 \}. \]
and define $KM(\kt)$ to be a connected component of $Aut(\kt) \setminus \pd(\kt)$.
The image of $Aut(\kt)$ in the orthogonal group $O(N(\kt))$ is known
by \cite{HuybrechtsMacriStellariOrientation2009}. In particular it is
an arithmetic subgroup.
Therefore, we can compactify the K\"ahler moduli space to a projective variety 
$\overline{KM}(\kt)$ using the Baily--Borel construction \cite{BailyBorel}.

The boundary $\overline{KM}(\kt) \setminus KM(\kt)$ consist of
components, called cusps, which are divided into the following types (cf.\ section \ref{sec:Ma}):
\begin{itemize}
  \item $0$-dimensional standard cusps,
  \item $0$-dimensional cusps of higher divisibility and
  \item $1$-dimensional boundary components.
\end{itemize}
In \cite{Ma1}, \cite{Ma2}  Shouhei Ma establishes a bijection between
\[   \left\{ \begin{array}{c} \text{ K3 surfaces $Y$ } \\ 
    \text{ with $\cd^b(Y) \isom \kt$ }\end{array} \right\}
\longleftrightarrow \left\{ \begin{array}{c} \text{ 
      standard cusps of the } \\ 
\text{  K\"ahler moduli space $KM(\kt)$ }
\end{array} \right\}. \tag{*} \]
Moreover, cusps of higher divisibility correspond to realizations of
$\kt$ as the derived category of sheaves on a K3 surface twisted by a
Brauer class.  Unfortunately the proof is not geometric but uses deep
theorems due to Mukai and Orlov to translate the statement into lattice
theory.

The aim of this work is to find a more geometric explanation for this
phenomenon using Bridgeland stability conditions
\cite{BridgelandStability}, \cite{BridgelandK3}.  The space
$Stab(\kt)$ of Bridgeland stability conditions on $\kt$ is a complex
manifold and carries canonical actions of $Aut(\kt)$ and of the universal
cover $\GLT$ of $\GL$.  
For each pair of $\omega \in Amp(X)$ and $\beta \in NS(X)_{\IR}$ with
$\omega^2>2$ Bridgeland constructs an explicit  stability conditions
$\sigma_{X}(\beta,\omega) \in Stab(\kt)$. Denote by
$Stab^\dagger(\kt)$ the connected component of $Stab(\kt)$ containing
these stability conditions.

A special open subset of $KM(\kt)$ can be identified with the quotient
space
\begin{align}\label{KM0_formula} 
  KM_0(\kt) \isom Aut^\dagger(\kt) \setminus Stab^\dagger(\kt) /\GLT, 
\end{align}
where $Aut^\dagger(\kt)$ is the group of auto-equivalences respecting
the distinguished component $Stab^\dagger(\kt)$. This statement is
essentially due to Bridgeland and was stated in \cite{Ma1} and
\cite{BridgelandSurvey} before.  However, it seems to rely on
properties of the group $Aut^\dagger(\kt)$ which are established in
appendix \ref{sec:aut_dagger}, cf. Corollary \ref{double_quot}.  We
denote the quotient map by $\pi: Stab^\dagger(\kt) \ra KM(\kt).$

\subsection{Hearts of stability conditions} 
Our first result addresses the following question:
Every stability condition $\sigma$ determines a heart $\ka(\sigma)$ of a bounded
t-structure. Also, every derived equivalence $\Phi:\kd^b(Y) \sra \kt$
determines the heart $\Phi(Coh(Y))$. How are these
two hearts related for $\sigma$ near the cusp  associated to $Y$?

\begin{thm*}[\ref{limiting_hearts}]
  Let $[v] \in \KM(\kt)$ be a standard cusp, and $Y$ the K3 surface
  associated to $[v]$ by (*).
  Then, there exists a path $\sigma(t) \in Stab^\dagger(\kt), t \gg 0$ 
  and an equivalence $\Phi: \cd^b(Y) \sra \kt$ such that
  \begin{enumerate}
  \item $\limt \pi(\sigma(t)) = [v] \in \KM(\kt)$ and
  \item $\limt \ka(\sigma(t)) = \Phi( Coh(Y) )$  as subcategories of $\kt$.
  \end{enumerate}
\end{thm*}

The path in this theorem is the image of $\sigma_Y(t \beta, t \omega)$
under a certain equivalence.  It is easy to construct other paths
satisfying (1) which have limiting hearts given by tilts of $Coh(Y)$.
The natural question arises, how all limiting hearts look like.

Instead of allowing all possible paths we identify a class of paths
$\gamma(t) \in KM(\kt)$, called {\em linear degenerations to a cusp}
$[v] \in \KM(\kt)$, and restrict our attention to them. The
prototypical example of a linear degeneration is $\pi(\sigma_X(\beta,
t \omega))$. In this case the heart of $\sigma_X(\beta,
t \omega)$ is constant and given by an explicit tilt of $Coh(X)$.
We prove the following proposition.

\begin{prop*}[\ref{linear_geodesics}, \ref{lin_deg_prop}]
  Let $[v] \in \KM(\kt)$ be a standard cusp and $\gamma(t) \in KM(\kt)$ be
  a linear degeneration to $[v]$, then  $\gamma(t)$ is a geodesic
  converging to $[v]$.
\end{prop*}
The (orbifold-)Riemannian metric we use is induced via the isomorphism 
$\pd(\kt) \isom O(2,\rho)/SO(2) \times O(\rho)$.

\begin{conj*}
  Every geodesic converging to $[v]$ is a linear degeneration.
\end{conj*}
This conjecture is true in the case that $X$ has Picard rank one.
Moreover, if one uses the Borel--Serre compactification to compactify
$KM(\kt)$ the conjecture seems to follow from \cite{JiMacPherson}.

The next theorem classifies paths of stability conditions mapping to 
linear degenerations in the K\"ahler moduli space.
\begin{thm*}[\ref{degeneration_thm}]
  Let $[v]$ be a standard cusp of $\KM(\kt)$. Let $ \sigma(t) \in
  Stab^\dagger(X)$ be a path in the stability manifold such 
  that $\pi(\sigma(t)) \in \KM(X)$ is a linear
  degeneration to $[v]$.  Let $Y$ be the K3 surface
  associated to $[v]$ by (*).  Then there exist
  \begin{enumerate}
  \item a derived equivalence $\Phi: \cd^b(Y) \xlra{\sim} \kt$
  \item classes $\beta \in NS(Y)_\IR$, $\omega \in \overline{Amp}(Y)$ and
  \item a path $g(t) \in \GLT$
  \end{enumerate}
  such that 
  \[ \sigma(t)  =  \Phi_* (\sigma_Y^*(\beta,t\, \omega) \cdot g(t))  \]
  for all $t \gg 0$.

  Moreover, the hearts of $\sigma(t) \cdot g(t)^{-1}$ are independent
  of $t$ for $t \gg 0$.
  If $\omega \in Amp(X)$, then the heart can be explicitly described
  as a tilt  of $Coh(Y)$.
\end{thm*}
Here, $\sigma^*_X(\beta,\omega)$ is an extension of Bridgeland's
construction of $\sigma_X(\beta,\omega)$ to the case that $\omega
\in \overline{Amp}(X)$ and $\omega^2>2$ (cf.\ Lemma
\ref{sigma_extension_lemma}).

\begin{figure}
  \includegraphics[width=7cm]{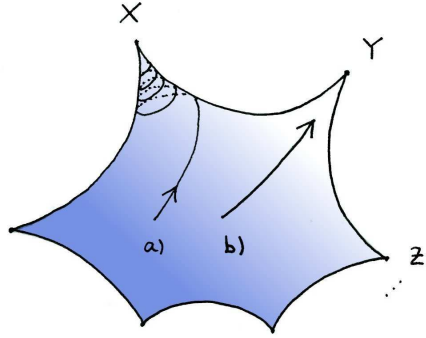}
  \caption{K\"ahler moduli space with cusps, associated K3 surfaces
    and two different degenerating paths.
    Paths of type a) are considered in Theorem \ref{limiting_hearts}, 
    whereas b) pictures a linear degeneration.}
\end{figure}

\subsection{Moduli spaces}
Another question we posed ourselves is: 
Can we construct $Y$ as a moduli space of stable objects in stability 
conditions near the associated cusp $[v] \in \KM(\kt)$?

For $v \in N(X)$ and $\sigma\in Stab(X)$ we consider the following
moduli space of semi-stable objects
\[ \km^\sigma(v) = \Set{ E \in \cd^b(X) }{ E \,
  \sigma\text{-semi-stable}, \, v(E)=v } / \sim \] where $E \sim F$ if
there is an even number $k \in 2\IZ$ and a quasi-isomorphism $E \isom
F[k]$. This is a version of the moduli stack constructed by Lieblich
\cite{Lieblich} and Toda \cite{Toda}. We prove the following result.

\begin{thm*}[\ref{reconstruction_thm}]
  If $v \in N(X)$ is an isotropic vector with $v.N(X)=\IZ$ 
  and $\sigma \in Stab^\dagger(X)$ a $v$-general stability condition, then:
  \begin{enumerate}
  \item The moduli space $\km^\sigma(v)$ is represented by a K3 surface $Y$.
  \item The Hodge structure $H^2(Y,\IZ)$ is the isomorphic to the subquotient 
    of $\widetilde{H}(X,\IZ)$ given by $v^\perp/ \IZ v$.
  \item The universal family $E \in \km^{\sigma}_X(v)(Y) \subset \cd^b(X \times Y)$ induces a derived 
    equivalence $\cd^b(X) \sra \cd^b(Y)$.
  \end{enumerate}
\end{thm*}

This is in some sense a negative answer to our question:
The isomorphism type of $\km^\sigma(v)$ does not depend on
whether the stability condition $\sigma$ is close to a cusp or not.  On
the other hand, the isotropic vector $v$ determines a standard cusp
$[v] \in \KM(\kt)$ and $Y$ is indeed the K3 surface associated to the
cups.

\subsection{Auto-equivalences}

On the way of proving the above result we need to construct enough
equivalences that respect the distinguished component $Stab^\dagger(\kt)$.
We collected our results in an appendix which is essentially independent
of the rest of the paper.

\begin{thm*}[\ref{sheaf_respect}, \ref{spherical_respect}, \ref{curve_respect}]
  The following equivalences respect the distinguished component.
  \begin{itemize}
    \item For a fine, compact, two-dimensional 
      moduli space of Gieseker-stable sheaves $M^h(v)$, the  
      Fourier--Mukai equivalence induced by the universal family.
    \item The spherical twists along Gieseker-stable spherical
      vector bundles.
    \item The spherical twists along $\ko_C(k)$ for a $(-2)$-curve
      $C \subset X$ and $k \in \IZ$.
  \end{itemize}
\end{thm*}

This allows us to show the following strengthening of a result
of \cite{HLOY2004}, \cite{HuybrechtsMacriStellariOrientation2009}.

\begin{prop*}[\ref{surj_of_aut}]
  Let $Aut^\dagger(\cd^b(X)) \subset Aut(\cd^b(X))$ be the subgroup of
  auto-equivalences which respect the distinguished component. 
  Then 
  \[ Aut^\dagger(\cd^b(X)) \lra O_{Hodge}^+(\HT(X,\IZ)) \]  
  is surjective.
\end{prop*}

Another direct consequence is the description (\ref{KM0_formula}) of
the K\"ahler moduli space, cf. Corollary \ref{double_quot}.

\subsection*{Acknowledgment}
 This work is part of a Ph.D. thesis written under the supervision of Prof. D. Huybrechts in Bonn
 whom we thank sincerely for his generous support.
 
 We thank T. Bridgeland, E. Looijenga,  N. Perrin, M. Rapoport, D. van Straten  for their help 
 with various questions 
 and  P. Sosna for his careful reading of earlier versions
 of this text.
 We thank the {\em Bonn International Graduate School in Mathematics} and the {\em Sonderforschungsbereich TR45} for financial support.


\subsection{Notation}
\label{sec:notation}
Our notation will largely follow Huybrechts' book \cite{HuybrechtsFM}.
Let $X$ be a projective K3 surface over the complex numbers. The
Picard rank of $X$ is denoted by $\rho(X)=rk(NS(X))$.

We write $\tilde{\HH}(X,\IZ)$ for the full cohomology 
$\HH^0(X,\IZ) \vsum \HH^2(X,\IZ) \vsum \HH^4(X,\IZ)$
endowed with the Mukai pairing
$ (r,l,s).(r',l',s') = l.l'-rs'-r's$, 
and the weight-two Hodge structure
\[ \HT^{1,1}(X)=\HH^{0,0}(X) \vsum \HH^{1,1}(X) \vsum \HH^{2,2}(X), \] 
\[ \HT^{2,0}(X)=\HH^{2,0}(X), \quad \HT^{0,2}(X)=\HH^{0,2}(X). \] 
We write $ N(X)=\HH^0(X,\IZ) \vsum NS(X) \vsum \HH^4(X,\IZ)$ for the
extended N\'eron--Severi group.  It is an even lattice of signature
$(2,\rho(X))$.

To a sheaf $A \in Coh(X)$ we associate the Mukai vector
\[ v(A) = \sqrt{\td(X)}.ch(A)=(r(A),c_1(A),s(A)) \in N(X), \]
where $s(A)=\frac{1}{2}c_1(A)^2-c_2(A)+r(A)$.
By the Riemann--Roch theorem we have $-\chi(A,B) = v(A).v(B).$
Therefore, we can identify $N(X)$ with the numerical Grothendieck
group $N(Coh(X))=K(Coh(X))/\mathrm{rad}(\chi)$ via the map 
$A \mapsto v(A)$. 

We denote by $\cd^b(X)$ the bounded derived category of coherent
sheaves on $X$.  We have natural isomorphisms between the numerical
Grothendieck groups
$N(\kd^b(X)) \isom N(Coh(X)) \isom N(X)$.
Every $\IC$-linear, exact equivalence $\Phi:\cd^b(X)
\ra \cd^b(Y)$  induces a Hodge isometry which we denote by
\[ \Phi^H: \HT(X,\IZ)\lra \HT(Y,\IZ).\] 
We say that two K3 surfaces
$X$ and $Y$ are derived equivalent if $\cd^b(X)$ is equivalent to
$\cd^b(Y)$ as a $\IC$-linear, triangulated category.


\section{Geometry of the Mukai lattice}\label{sec:Lattice}
Let $X$ be a projective K3 surface over the complex numbers and
$\kd^b(X)$ its derived category. Let $N=N(\kd^b(X))$ be the 
numerical Grothendieck group of $\kd^b(X)$.
In this section, we will introduce various groups and spaces that are
naturally associated to the lattice $N$.

The isomorphism $N \isom N(X)$ gives us the following 
extra structures.
\begin{enumerate}
\item An isotropic vector $v_0=(0,0,1) \in N$.
\item An embedding of a hyperbolic plane \[\vphi: U\isom H^0(X,\IZ)
  \vsum H^4(X,\IZ) \ra N.\]
\item The choice of an ample chamber $Amp(X) \subset NS(X)_\IR \subset N_\IR$.
\item A weight-two Hodge structure $\HT(X,\IZ)$ with $\HT^{1,1}(X) \cap \HT(X,\IZ)=N$.
  In particular, a group action of $O_{Hodge}(\HT(X,\IZ))$ on $N$.
\end{enumerate}
We will pay special attention to which constructions depend on what
additional data.

\begin{conv}\label{conv}
  In later sections, when we have fixed an identification $N=N(X)$, we will
  allow ourselves to abuse the notation by filling in the standard choices of
  the above extra structures. For example, we shall write $\kl(X)$ for the space
  $\kl(N(X),v_0,Amp(X))$ introduced in Definition \ref{plus_defn}.
\end{conv}

\newcommand{\pl}{\theta}
\subsection{The K\"ahler period domain}

Let $N$ be a non-degenerate lattice of signature $(2,\rho)$. 

\begin{defn}
  We define the {\em K\"ahler period domain} to be
  \[ \pd(N) = \Set{ [z] \in \IP(N_\IC) }{ z^2 =0, \; z.\bar{z} > 0} \subset \IP(N_\IC). \]
  We also introduce the following open subset of $N_\IC$:
  \[ \kp(N) =  \Set{ z \in N_\IC }{ \IR\<Re(z),Im(z)\> \subset N_\IR\, 
    \text{ is a positive $2$-plane} } \subset N_\IC. \]
  This set carries a natural free $Gl_2(\IR)$-action by identifying $N_\IC = N \tensor_\IZ \IC$  
  with $N \tensor_\IZ \IR^2$.
\end{defn}

\begin{lem}\label{grass_lemma}
  There is a canonical map
  \[ \pl :\kp(N) \lra \pd(N) \]
  which is a principal $\GL_2^+$-bundle.
\end{lem}
\begin{proof}
  This map is most easily described using the canonical isomorphism $\pd(N)
  \isom Gr^{po}_2(N_\IR)$, where $Gr^{po}_2(N_\IR)$ is the Grassmann manifold of
  positive definite, oriented two-planes in $N_{\IR}$ (cf.\ \cite[VII.\
  Lem.1]{Beauville1985}).  We define $\pl$ to map a vector $[z] \in \kp(N)$ to
  the oriented two-plane $P=\IR\<Re(z),Im(z) \>$. As $\GL$ acts simply and
  transitively on the set of oriented bases of $P$, this map is a principal
  $\GL$-bundle.
\end{proof}

In the case $N=N(X)=H^0(X) \vsum NS(X) \vsum H^4(X)$ there is a well known {\em tube model} of 
the period domain, given by
\[ exp\colon \{ z = x + i y \in NS(X)_\IC \,|\, y^2 > 0 \} \xlra{\isom} \pd(N),\; 
z \mapsto [(1,z,\half z^2 )].   \]
To define this map we used the full information about the embedding $U \isom H^0(X) \vsum H^4(X)$
into $N(X)$. In this section we will construct a similar map, which only depends on the isotropic 
vector $v_0=(0,0,1)$. Compare also \cite[Sec.\ 4]{Dolgachev1996}.

Let $N$ be a non-degenerate lattice of signature $(2,\rho), \rho \geq
1$. To a primitive isotropic vector $v \in N(X)$ we associate the lattice 
\[ L(v) = v^\perp / \IZ v = \Set{z \in N }{ z.v = 0 } / \IZ v \] of
signature $(1,\rho-1)$ and the affine space $ A(v) = \Set{z \in N}{
  z.v = -1 } / \IZ v $ over $L(v)$.  Note that, if $N=N(X)$ and
$v=v_0$, then $L(v) \isom NS(X)$.

\begin{defn}
  We define the { \em tube domain } associated to $N$ and $v$ as 
  \[ T(N,v)= A(v)_\IR \times C(L(v)) \]
  where  $ C(L(v))  = \Set{y \in L(v)_\IR }{y^2 > 0 }$.
  Note that $A(v)_\IR$ is naturally an affine space over $L(v)_\IR$.
  We consider $T(N,v)$ as a subset of $N_\IC/ \IC v$ by mapping $(x,y)$ 
  to $x+iy  \in N_\IC / \IC v.$
  We will often write $x+iy$ for a pair $(x,y) \in T(N,v)$.
\end{defn}

\begin{lem}\label{tube_model}
  There is a canonical map $Exp_v:  T(N,v) \ra  \kp(N)$  such that 
  \[ exp_v = \theta \circ Exp_v: T(N,v) \lra \pd(N) \] 
  is an isomorphism. 
\end{lem}
\begin{proof}
  We construct the inverse to $Exp_v$. The set
  \[ Q(v) = \Set{ z \in N_\IC }{z^2 =0, z.\bar{z}>0, z.v=-1 } \subset \kp(N) \]
  is a section for the $\GL$-action on $\kp(N)$. 

  One checks immediately, that the projection $N_\IC \ra N_\IC/ \IC v$ 
  induces an isomorphism $Q(v) \ra T(N,v)$. Define $Exp_v$ to be the inverse 
  of this isomorphism.
\end{proof}

\begin{rem}\label{q_v_remark}
  In particular, we obtain a section of the $\GL$-bundle $\theta$,
  namely $q_v=Exp_v \circ exp_v^{-1}: \pd(N) \ra \kp(N)$.
\end{rem}

\begin{lem}\label{equivariance}
  Let $g \in O(N)$ be an isometry of $N$, then $g$ induces a commutative diagram
  \[ \xymatrix{ 
    T(N,v) \ar[d]_g \ar[r]^{Exp_v} & \kp(N) \ar[d]^g\ar[r]^\pl & \pd(N) \ar[d]^g  \\ 
    T(N,w)          \ar[r]^{Exp_w} & \kp(N)       \ar[r]^\pl & \pd(N),
  } \]
  where $w=g \cdot v$. \hfill \qed
\end{lem}

\subsection{Roots, Walls and Chambers}

Recall that, every lattice $N$ determines a {\em root-system} $ \Delta(N)=\Set{ \delta \in N
}{ \delta^2 = -2 }. $ To every root $\delta \in \Delta(N)$ there is an
associated {\em reflection}, $s_\delta: w \mapsto w + (\delta.w) \delta$ which
is an involutive isometry.  The subgroup $W(N) \subset O(N)$ generated by the
reflections is called {\em Weyl group}.

If we are given an isotropic vector $v \in N$ we define
\[ \NDel(N,v) = \Set{ \delta \in \Delta(N) }{ -v.\delta > 0 },
\quad \Delta^0(N,v) = \Set{ \delta \in \Delta(N) }{ v.\delta = 0 }. \]

The group generated by the reflections $\Set{s_\delta}{ \delta \in
  \Delta^0(N,v)}$ is denoted by $W^0(N,v)$.

\begin{defn}
  To $\delta \in \Delta(N)$ we associate a divisor
  \[ D(\delta)=\Set{[z]}{z.\delta = 0} \subset \pd(N) \]
  and define $ \pd_0(N) = \pd(N) \setminus \bigcup \Set{D(\delta)}{\delta \in \Delta(N)} . $
\end{defn}
The connected components of $\pd(N) \isom T(N,v)$ are clearly
contractible. In contrast the space $\pd_0(N)$ is the complement of a
infinite number of hypersurfaces and will therefore in general not
even have a finitely generated fundamental group.  Following
Bridgeland (cf.\ Remark \ref{U_boundary_rem}) we will decompose
$\pd_0(N)$ into a union of codimension-one submanifolds called {\em
  walls} and their complements called {\em chambers} such that each
individual chamber is contractible.

\begin{defn}\label{wall_notation}
  Given $\delta \in \NDel(N,v)$ and a primitive, isotropic vector $v \in N$, we define
  a real, codimension one submanifold, called  {\em wall}
  \[ W_A(\delta,v) = \Set{ [z] \in \pd(N) }{ - z.\delta / z.v \in \IR_{\leq 0} } \subset \pd(N). \]
  To a vector $\delta \in \Delta^0(N,v)$ we associate the wall
  \[ W_C(\delta,v) = \Set{ [z] \in \pd(N) }{  - z.\delta / z.v \in \IR }. \]
  One can check, that $W_C(\delta,v)$ only depends on the image $l$ of $\delta$ 
  in $L(v)=v^{\perp}/\IZ v$. Therefore we write also $W_C(l,v)$ for $W_C(\delta,v)$.
  We define
  \begin{align*} 
    \pd_{A}(N,v) &= \pd(N) \setminus \bigcup \Set{W_A(\delta,v)}{\delta \in \NDel(N,v)} \\
    \pd_{C}(N,v) &= \pd(N) \setminus \bigcup \Set{W_C(\delta,v)}{\delta \in \Delta^0(N,v)}.
  \end{align*}
  We denote the intersections $\pd_{0}(N) \cap \pd_{A}(N,v), \pd_{A}(N,v) \cap \pd_{C}(N,v),$ etc.\
  by $\pd_{0,A}(N,v),\pd_{A,C}(N,v),$ etc., respectively. 
  For any combination $*$ of the symbols $0,A,C$ we set
  \[ T_*(N,v)=exp_v^{-1}(\pd_*(N,v)). \]
\end{defn}

\begin{rem}
  The seemingly unnatural notation, $-z.\delta/z.v \in \IR_{\leq 0}$, is 
  chosen since for $z=Exp_v(x + i y)$ we have $z.v  = -1$, and hence
  $-z.\delta/z.v=z.\delta$. 
  
  The sets considered above are indeed complex manifolds as the unions 
  of $D(\delta), W_A(\delta,v)$ and $W_C(l,v)$ are locally finite
  \cite[Lem.\ 11.1]{BridgelandK3}.

  If $\delta \in \Delta^0(N,v)$, then $D(\delta) \subset W_C(\delta,v)$, and if
  $\delta \in \pm \Delta^{>0}(N,v)$, then $D(\delta) \subset W_A(\delta,v)$.
  Therefore, \[ \pd_{A,C}(N,v) \subset  \pd_{0}(N,v). \]
\end{rem}

\begin{lem}\label{L_crit} \cite[Lem.\ 6.2, Lem.\ 11.1]{BridgelandK3}
  Let 
  \[ \pd_{>2}(N,v) = \Set{exp_v(x + i y) \in \pd(N) }{ y^2 > 2 } \subset \pd(N)  \]
  and denote by $\pd_{0,>2}(N,v)=\pd_{>2}(N,v) \cap \pd_{0}(N,v)$ etc.\ the various intersections.

  Then \[ \pd_{>2}(N,v) \subset \pd_{A}(N,v) \]
  and the inclusions
  $ \pd_{>2}(N,v) \subset \pd_{A}(N,v)$, $\pd_{0,>2}(N,v) \subset
  \pd_{0,A}(N,v)$ and  $\pd_{0,C,>2}(N,v) \subset \pd_{0,C,A}(N,v)$ 
  are deformation retracts.
\end{lem}

Recall that, $L=L(v)$ is a lattice of signature $(1,\rho -1)$.
We defined $C(L)=\Set{y \in L_\IR }{ y^2 > 0 }$ to be the {\em positive cone}.
This space has two connected components, let $C(L)^+$ be one of them.
For $l \in \Delta(L)$ define a {\em wall} $W(l)=\Set{ y \in C(L) }{ y.l = 0 }$ and set 
$C(L)_0 = C(L) \setminus \bigcup \Set{W(l)}{l \in \Delta(L) }.$
Connected components of $C(L)_0$ are called {\em chambers}.

\begin{lem}
  We have
  \[  T_{C}(N,v) = A(v)_{\IR} \times C_0(L(v)). \]
  Moreover, the connected components of $\pd_{0,A,C}(N,v) \subset \pd_0(N)$ are contractible. 
\end{lem}
\begin{proof}
  The first assertion is a direct calculation.  It follows, that
  the connected components of $\pd_{C}(N,v)$ and $\pd_{C,>2}(N,v)$
  are contractible.  Now, the components of $\pd_{0,A,C}(N,v)$ are
  contractible since Lemma \ref{L_crit} shows that $\pd_{C,>2}(N,v)
  \subset \pd_{0,A,C}(N,v)$ is a deformation retract.
\end{proof}

\begin{defn}\label{plus_defn}
  If we are given a chamber $Amp \subset C(L(v))_0$ we
  define $\kl(N,v,Amp) \subset \pd_{0,A,C}(N,v)$ to be the connected component
  containing the vectors $exp_v(x + iy)$ with $y \in Amp, y^2 > 2$.

  The connected component of $\pd(N)$ containing $\kl(N,v,Amp)$ is denoted 
  by $\pd^+(N)$. We introduce also the notation 
  $\pd^+_{*}(N) = \pd^+(N) \cap \pd_{*}(N,v)$ for a combination $*$ of the symbols $0,A,C,>2$.

  The orthogonal group $O(N)$ acts on $\pd(N)$. 
  We denote by $O^+(N)$ be the index two subgroup preserving the connected
  components of $\pd(N)$.
\end{defn}

\begin{rem}
  The set $\kl(N,v,Amp)$ can be described more explicitly as
  $\Set{ exp(x+ iy) \in \pd(X) }{ y \in Amp(X), (*) }$
  where $(*)$ is the condition
  \[ Exp(x+iy).\delta \notin \IR_{\leq 0} \qtext{for all} \delta \in \Delta^{>0}(N(X),v).\]
  This is the description used in \cite{BridgelandK3}.
\end{rem}


\section{Ma's Theorem}\label{sec:Ma}
The goal of this section is to explain Ma's theorem about cusps of the
K\"ahler moduli space of a K3 surface (\cite{Ma1}, \cite{Ma2}).

We use the recent result \cite{HuybrechtsMacriStellariOrientation2009}
to make the construction of the K\"ahler moduli space intrinsic to the
derived category. This allows us to formulate Ma's theorem in a more
symmetric way.

\subsection{The K\"ahler moduli space}\label{subsec:intrinsic}
Recall from Definition \ref{plus_defn}, that $\pd(X)=\pd(N(X))$ has a distinguished connected 
component $\pd^+(X)$ containing the vectors $exp(x+iy)$ with $y \in NS(X)$ ample.
The key ingredient for our construction of $KM(\kt)$ is the following theorem.
\begin{thm}\cite{HLOY2004},\cite{PloogPHD}, \cite[Cor.\ 4.10]{HuybrechtsMacriStellariOrientation2009}  \label{orientation} \label{Or}
  The image of
  \[ Aut(\cd^b(X)) \lra O_{Hodge}(\tilde{H}(X,\IZ)) \]
  is the index-two subgroup $O_{Hodge}^+(\tilde{H}(X,\IZ))$ of isometries preserving 
  the component $\pd^+(X)\subset \pd(X)$.

  Let $\Phi:\cd^b(X) \ra \cd^b(Y)$ be a derived equivalence between two K3 surfaces.
  Then the isomorphism $\Phi^H: \pd(X) \ra \pd(Y)$ maps $\pd^+(X)$ to $\pd^+(Y)$.
\end{thm}

This theorem allows us to make the following definition.
\begin{defn}
  Let $\pd^+(\kt)$ be the connected component of the period domain $\pd(\kt)=\pd(N(\kt))$
  which is mapped to $\pd^+(X)$ under every derived equivalence $\kt \isom \cd^b(X)$.

  We define the {\em K\"ahler moduli space } of $\kt$ to be
  \[ KM(\kt) = \Gamma_{\kt} \setminus \pd^+(\kt) \]
  where $\Gamma_{\kt}$ is the image of $Aut(\kt)$ in $O(N(\kt))$.
\end{defn}

\begin{rem}
  Let us introduce the notation $KM(X)$ for $KM(\cd^b(X))$.
  Theorem \ref{Or} shows, that we have a canonical isomorphism
  \[ KM(X) \isom \Gamma_X^+ \setminus \pd^+(X), \]
  where $\Gamma_X^+ \subset O(N(X))$ is the image of $O^+_{Hodge}(\HT(X,\IZ))$ in $O(N(X))$.
  Ma works in the setting $\kt=\kd^b(X)$ and uses $\Gamma_X^+ \setminus \pd^+(X)$
  as definition for the K\"ahler moduli space.
\end{rem}

\begin{rem}\label{KM_stab_constr}
  There is another construction of the K\"ahler moduli space using the theory of 
  Bridgeland stability conditions which is proved in the appendix 
  cf.\ Corollary \ref{double_quot}:
  \[
  KM_0(X) \isom Aut^\dagger(\cd^b(X)) \setminus Stab^\dagger(X) / \GLT .
  \]
  Here $KM_0(X) = \Gamma^+_X \setminus \pd^+_0(X)  \subset KM(X)$ is the complement of
  a divisor and $Aut^\dagger(\cd^b(X))$ is the group of auto-equivalences respecting 
  the distinguished component $Stab^\dagger(X)$ of the stability manifold. 
  This was also stated in \cite{Ma1} without proof.

  Note that, this description is not intrinsic to the derived category
  as the component $Stab^\dagger(X) \subset Stab(\cd^b(X))$ may a
  priori depend on $X$.  But in fact, no other component of the
  stability manifold is known.
\end{rem}

\begin{ex}
  If $X$ has Picard rank $\rho(X)=1$ and the ample generator $H \in NS(X)$ has square $H.H=2n$,
  then the K\"ahler moduli space is isomorphic to a Fricke modular 
  curve $KM(X) \isom  \Gamma_0^+(n)  \setminus \IH$. 
  See \cite[Sec.\ 5]{Ma2}, \cite[Thm.\ 7.1]{Dolgachev1996}.
\end{ex}

The subgroup $\Gamma_\kt \subset O(N(\kt))$ is of finite index since it contains
\[ O_0^+(N(\kt)) = O^+(N(\kt)) \cap Ker(O(N(\kt)) \ra Aut(A(N(\kt))))  \] 
where $A(N(\kt))=N(\kt)\dual / N(\kt)$ is the discriminant group, cf.\ \cite[Def.\ 3.1.]{Ma1}.
Hence we can apply a general construction of Baily and Borel to compactify the
K\"ahler moduli space.
\begin{thm}[Baily--Borel]  \cite{BailyBorel}
  There is a natural compactification $KM(\kt)\subset \overline{KM}(\kt)$
  which is a normal, projective variety over $\IC$.

  The boundary $\del KM(\kt)= \overline{KM}(\kt) \setminus KM(\kt)$ consists of 
  zero- and one-dimensional components called {\em cusps}, which are in bijection 
  to $\Gamma_\kt \setminus \kb_i $, where
  \[ \kb_i = \Set{ I \subset N(\kt) }{I \text{ primitive, isotropic, } rk(I)=i+1} \]
  for $i=0,1$ respectively.
\end{thm}

\begin{defn}\label{def_standard_vector}
  The set of zero-dimensional cusps is divided further with respect to divisibility. 
  For $I \in \kb_0$ we define 
  \[ div(I)=g.c.d(\Set{ v.w }{ v \in I, w \in N(\kt)}) \] 
  and set $\kb_0^d=\Set{ I \in \kb_0  }{ div(I)=d }$.
  Cusps corresponding to elements of $\kb_0^1$ are called {\em standard cusps}.   
  
  We call $v \in N$ a {\em standard vector}\footnote{This definition is not standard.} 
  if $v.v=0$ and $div(v):=div(\IZ v)=1$.   
\end{defn}  

\begin{rem}\label{std_vectors}
  The group $\Gamma_\kt$ contains the element $-id_{N(\kt)}=[1]^H$ which interchanges the 
  generators of any $I \in \kb_0^1$.
  Therefore, the map $v \mapsto \IZ v$ induces a bijection
  \[ \Gamma_\kt \setminus \Set{ v \in N(\kt) }{ v \text{ standard } } 
  \isom \{\text{ standard cusps of $\KM(\kt)$}\ \}. \]  
  We will refer to standard cusps as equivalence classes $[v]=\Gamma_\kt \cdot v$ 
  via this bijection.
\end{rem}

\subsection{Ma's theorem}\label{subsec:Ma}

\begin{defn}
  The K\"ahler moduli space of $\kd^b(X)$ comes with a distinguished standard cusp 
  $[v_0] \in \KM(X), v_0=(0,0,1) \in N(X)$, which is called  {\em large volume limit}.
\end{defn}

We are now ready to state Ma's theorem.
\begin{thm}[Ma] \cite{Ma1}, \cite{Ma2} \label{Ma_Thm}
  There is a canonical bijection
  \[ \Set{ X \text{ K3 surface }}{ \kt \isom \cd^b(X) }_{/ \isom}
  \longleftrightarrow 
  \{ \text{ standard cusps of } \KM(\kt) \}. \]
  The cusp of $\KM(\kt)$ associated to  $Y$ corresponds to the large volume limit of $X$ 
  under the isomorphism $\KM(\kt) \isom \KM(X)$ induced by any equivalence $\kt \isom \cd^b(X)$.

  We denote the K3 surface associated to a cusp $[v]$ by $X(v)$.
\end{thm}

\begin{proof}
  We sketch Ma's original proof for the case $\kt = \kd^b(X)$  
  and then generalize to our situation.

  Every derived equivalence $\Phi:\cd^b(Y) \ra \cd^b(X)$ induces an isometry
  $\Phi^H: N(Y) \ra N(X)$, and therefore an embedding of the hyperbolic plane
  \[   U \isom H^0(Y) \vsum H^4(Y) \subset N(Y) \xlra{\Phi^H} N(X). \]
  It follows from Orlov's derived global Torelli theorem 
  \cite[Prop.\ 10.10]{HuybrechtsFM} that this construction induces a bijection
  \begin{align*} 
    \Set{ Y \text{ K3 surface}}{\cd^b(X) \isom \cd^b(Y)}_{/ \isom} \lra  Emb(U,N(X))/  \Gamma_X,
  \end{align*}
  where $Emb(U,N(X))$ is the set of all embeddings of the hyperbolic plane $U$ into $N(X)$,
  and $\Gamma_X \subset O(N(X))$ is the image of $O_{Hodge}(\HT(X,\IZ))$ in $O(N(X))$.
  The key insight of Ma is that the map $\vphi \mapsto \vphi(f)$, where $e,f \in U$ is the standard 
  basis, induces a bijection
  \[ 
   Emb(U,N(X)) / \Gamma_X  \lra \Set{ v \in N(X) }{ v \text{ standard } }/ \Gamma_X^+.  
  \]
  Combining with Remark \ref{std_vectors} one gets a bijection
  \[  \Set{ Y \text{ K3 surface}}{\cd^b(X) \isom \cd^b(Y)}_{/ \isom} 
  \lra \{ \text{ standard cusps of $\KM(X)$ }\} \]
  which maps $X$ maps to $[v_0]$.
  Note that, the Hodge structure $H^2(Y,\IZ)$ of a K3 surface $Y$ 
  can be reconstructed from the associated cusp $[v]$
  as the subquotient $v^\perp / v$ of $\HT(X,\IZ)$.

  To generalize to arbitrary $\kt$ we choose an equivalence $\kt \isom
  \cd^b(X)$ and claim that the above bijection is independent of this choice.
  Indeed, if we are given another equivalence $\kt \isom \cd^b(Y)$, then the composition
  $\Phi: \cd^b(X) \isom \kt \isom \cd^b(Y)$ induces a Hodge isometry $\Phi^H:\HT(X,\IZ) \ra \HT(Y,\IZ)$.
  If a standard vector $v \in N(\kt)$ corresponds to $v_1 \in N(X)$ and $v_2 \in N(Y)$ 
  then $\Phi^H$ induces an isomorphism of Hodge structures 
  \[ H^2(X(v_1),\IZ) \isom v_1^\perp/v_1 \lra v_2^\perp / v_2 \isom H^2(X(v_2),\IZ). \]
  Now, the global Torelli theorem shows that the K3 surfaces $X(v_1)$ and $X(v_2)$ are isomorphic. 
\end{proof}

\begin{rem}\label{fm_criterion}
  Let $v \in N(X)$ be a standard vector defining a standard cusp of $\KM(X)$.
  The  Fourier--Mukai partner $Y=X(v)$ associated to this cusp via Ma's theorem
  is determined up to isomorphism, by the property that
  \[ H^2(X(v),\IZ) \isom v^\perp / v \]
  as subquotient of $\HT(X,\IZ)$.
  
  We cannot formulate an analogues statement for the cusps of $\KM(\kt)$ since there is 
  no construction of the Hodge structure $\HT(X,\IZ)$ known, which is intrinsic to 
  the category $\kd^b(X)$.
\end{rem}

\begin{rem}
  Let $X$ be a K3 surface, $[v] \in \overline{KM}(X)$ a standard cusp and let $Y=X(v)$
  be the associated Fourier--Mukai partner.  We have seen that every derived equivalence 
   $\Phi: \cd^b(X) \sra \cd^b(Y)$ maps $[v]$ to the 
  large volume limit $[v_0] \in \overline{KM}(Y)$.

  In appendix \ref{sec:reduction_to_lvl} we will strengthen this result in two directions.
  Firstly, we will construct a $\Phi$, with the property that $\Phi^H$ maps $v$ to $v_0$ 
  and not only the orbit $[v]$ to $[v_0]$.
  Secondly, the equivalence $\Phi$ respects the distinguished component of the stability manifold 
  (cf.\ Subsection \ref{sec:constr-stab-cond}).
\end{rem}


\section{Stability conditions}\label{sec:Stability}
Our next goal is to relate the K\"ahler moduli space to the stability manifold.
In this section we recall from \cite{BridgelandStability} and
\cite{BridgelandK3} the basic theory of Bridgeland stability conditions in the
special case of a K3 surface. On the way we introduce the notation and establish
some geometric results that will be used in sequel.  The link to the
K\"ahler moduli space will be made in section \ref{sec:Degenerations} and
Corollary \ref{double_quot} in the appendix.

\subsection{Definition of stability conditions}

\renewcommand{\mho}{z} 
Let $X$ be a K3 surface.  Recall from \cite[Def.\ 5.7, Def.\ 2.3.,
Prop.\ 5.3]{BridgelandStability}, that a {\em stability condition}
$\sigma$ on $\cd^b(X)$ consists of
\begin{enumerate}
\item a heart $\ka$ of a bounded t-structure on $\cd^b(X)$ and
\item a vector $\mho \in N(X)_{\IC}$ called {\em central charge}
\end{enumerate}
with the property that $Z\colon K(\ka) \ra \IC, A \mapsto v(A).\mho$
satisfies
\[ Z(A) \in \IH \union \IR_{<0} \text{\; for all \;} A \in \ka, \, A
\neq 0. \] We require moreover local-finiteness and the existence of
Harder--Narasimhan filtrations.

In the usual definition, the datum of the heart is replaced by a
collection of subcategories $\kp(\phi) \subset \ka, \phi \in \IR,$
called \textit{slicing}. The equivalence to the above definition was shown
in \cite[Prop.\ 5.3]{BridgelandStability}.

The main result about the stability manifold of a K3 surface is the following theorem.

\begin{thm}\cite[Cor.\ 1.3]{BridgelandStability}, \cite[Thm.\
  1.1]{BridgelandK3} \label{stability_theorem} 
  The set of all stability conditions on a K3 surface $X$ has 
  the structure of a (finite-dimensional) complex manifold $Stab(X)$.

  There is a distinguished connected component $Stab^\dagger(X)$ of
  $Stab(X)$ such that the map $\sigma=(\ka,\mho) \mapsto \mho$ induces 
  a Galois cover
  \[ \pi: Stab^\dagger(X) \lra \kp_0^+(X).  \] 
  Moreover, the Galois group is
  identified with the $Aut^{\dagger}_0(\cd^b(X)) \subset Aut(\cd^b(X))$, the
  group of auto-equivalences that respect the component $Stab^\dagger(X)$ and
  act trivially on the cohomology $\HT(X,\IZ)$.
\end{thm}

\subsection{Group actions}

Given a derived equivalence $\Phi: \cd^b(X) \sra \cd^b(Y)$ and a stability
condition $(\ka,z)$ on $\cd^b(X)$ we get an induced stability condition 
$\Phi_* (\ka,\mho)=(\Phi(\ka), \Phi^H(\mho))$ on $\cd^b(Y)$.
In this way we obtain a left action of the group $Aut(\cd^b(X))$ on $Stab(X)$.

There is also a right action of the group 
\[ \GLT = \Set{(T,f)}{T \in \GL,\, f:\IR \ra \IR, f(\phi + 1)=f(\phi)+1 \stext{with} (\#) }\]
on $Stab(X)$  (cf. \cite[Lem.\ 8.2.]{BridgelandStability}). Here $(\#)$ stands
for the condition
$\IR_{>0} \, T \cdot exp(i \pi \phi) = \IR_{>0} \, exp(i \pi f(\phi))$.

\begin{ex}\label{shift_example}
  For $\lambda \in \IR$ set 
  $ \Sigma_\lambda=(\exp(i \pi \lambda),\phi \mapsto \phi + \lambda) \in \GLT $.
  Then the action of the shift $[1]$ equals the action of the 
  $\Sigma_1$ on the stability manifold.
\end{ex}

\subsection{Construction of stability conditions}\label{sec:constr-stab-cond}
Explicit examples of stability conditions on a K3 surface are constructed as follows.

Fix classes $\beta \in NS(X)_\IR$ and $\omega \in Amp(X)$ and define a central charge 
\[ Exp(\beta + i \omega)=(1,\beta + i \omega,\half (\beta + i \omega)^2 ) \in N(X)_\IC. \]
For a torsion free sheaf $A$ of positive rank denote by
\begin{align*}
 \mu_\omega^{min}(A) &= \mathrm{inf}\, \{ \mu_\omega(Q) \,|\, A \surj
 Q,\, Q \text{ torsion free} \}  \\
 \mu_\omega^{max}(A) &= \mathrm{sup}\, \{ \mu_\omega(S) \,|\, S \inj  A \}
\end{align*}
the extremal slopes. 
Define full subcategories of $Coh(X)$ by
\begin{align*}
 \kt &= \{ A  \in Coh(X) \,|\, A \, \text{torsion or} \, \mu_\omega^{min}(A/A_{tors}) > \beta.\omega \} \\
 \kf &= \{ A \in Coh(X) \,|\, A \, \text{torsion free and } \, \mu_\omega^{max}(A) \leq \beta.\omega \}.
\end{align*}
The following full subcategory of $\cd^b(X)$ is a heart of a bounded t-structure.
\[  \ka(\beta,\omega) = 
\{ E \in \cd^b(X) \,|\, H^0(E) \in \kt,\, H^{-1}(E) \in \kf,\, H^{i}(E)=0 \;\text{if}\; i \neq 0,-1 \}. 
\]

\begin{thm}\label{Existence} \cite[Lem.\ 6.2, Prop.\ 11.2]{BridgelandK3}
  The pair 
  \[ \sigma(\beta,\omega)=(\ka(\beta,\omega),z=Exp(\beta + i \omega)) \]
  is a stability condition on $\cd^b(X)$ if $\theta(\mho)=exp(\beta + i \omega) \in \kl(X) \subset
  \pd(X)$.

  The set of all stability conditions arising in this way is denoted by $V(X)$.
\end{thm}

The connected component of $Stab(X)$ containing $V(X)$ is called {\em
  distinguished component} and denoted by $Stab^\dagger(X)$.

Let $\Phi: \cd^b(X) \ra \cd^b(Y)$ be a derived equivalence between two
K3 surfaces.  We say $\Phi$ {\em respects the distinguished component}
if $\Phi_* Stab^\dagger(X)=Stab^\dagger(Y).$

\begin{rem}\label{hearts}
  The heart $\ka(\beta, \omega \lambda)$ is independent of $\lambda > 0$. 
  Indeed, we have $\mu_{\lambda \omega}(A)=\lambda \mu_{\omega}$ and hence
  the conditions $\mu_{\omega}^{min}(A/A_{tors}) > \beta. \omega$ and 
  $\mu_\omega^{max}(A/A_{tors}) \leq \beta.\omega$ are invariant 
  under $\omega \mapsto \lambda \omega$.
  Therefore $\kt$ and $\kf$ do not depend on $\lambda$.
\end{rem}

\begin{rem}\label{U_definition}
  By \cite[Prop.\ 10.3]{BridgelandK3} the action of $\GLT$-action on $V(X)$ is free.
  We introduce the notation  $U(X):=V(X) \cdot \GLT \isom V(X) \times \GLT$ 
  for the image.
\end{rem}

The following proposition gives an important characterization of $U(X)$.
\begin{prop}\cite[Def.\ 10.2, Prop.\ 10.3]{BridgelandK3}\label{stab_criterion}
  Let $\sigma=(\ka,\mho)$ be a stability condition on $\cd^b(X)$.
  Then $\sigma \in U(X)$ if and only if the following properties hold.
  \begin{enumerate}
  \item All skyscraper sheaves $\ko_x$ are stable of the same phase.
  \item The vector $\mho$ lies in $\kp_0(X)$.
  \end{enumerate}
\end{prop}

\subsection{Geometric refinements}

\begin{rem}\label{stab_diagram}
  Let us summarize the above discussion in the following diagram:
  \[ \xymatrix@R=0.7pc{
   U(X) \ar[ddd] \ar@{^(->}[r]_{\text{open}} & Stab^\dagger(X)
   \ar[rr]^{\pi} \ar[dd]_{\GLT} 
   && \kp_0^+(X) \ar[dd]^{\pl}_{\GL}  \\ \\
   & Stab^\dagger(X)/\GLT \ar[rr]^{\bar{\pi}}         && \pd_0^+(X)  \\ 
    V(X) \ar@{^(->}[ur]_{\text{open}}  \ar@{<->}[rrrr]^{\isom}& && & \kl(X)  \ar@{_(->}[ul]^{\text{open}} 
  } \]
  Here we identify $V(X)$ with its image in $Stab^\dagger(X)/\GLT$. The maps $\pi,\bar{\pi}$
  are covering spaces. Moreover, the map $\pi: U(X) \ra \theta^{-1}(\kl(X))$
  is a covering space with fiber $\IZ$.
\end{rem}

\begin{lem}\label{sigma_extension_lemma}
  Consider the map $\sigma: \kl(X) \ra V(X)$ which maps $exp(\beta + i
  \omega)$ to the stability condition $\sigma(\beta,\omega)$.

  Let $\overline{\kl}_{0,>2}(X)$ be the closure of $\kl_{0,>2}=\kl(X) \cap \pd_{0,>2}^+(X)$ in $\pd_{0,>2}^+(X)$,
  and  $\overline{V}_{>2}(X)$ be the intersection of $\overline{V}(X)$ with $\pi^{-1}(\pd_{>2}(X))$.

  Then there is a unique continuous extension of $\sigma|_{\kl_{0,>2}}(X)$ to an isomorphism 
  \[ \sigma^*:\overline{\kl}_{0,>2}(X) \sra \overline{V}_{>2}(X).\] 
\end{lem}

\begin{proof}
  Under the isomorphism $exp: T(N(X),v_0) \ra \pd(X)$ the set
  $\overline{\kl}_{0,>2}(X)$ gets identified with 
  $ \Set{ x + i y }{ y \in \overline{Amp}(X), \; y^2 > 2, \; (*) }$ 
  where $(*)$ is the condition
  \[ ( x + i y ).l \notin \IZ \qtext{for all} l \in \Delta(NS(X)) \]
  As we have $y.l \geq 0$ for $y \in \overline{Amp}(X)$, we can
  retract $\overline{\kl}_{0,>2}(X)$ into the subset
  $\kl(X)_{>2}=\kl(X) \cap \pd_{>2}$ via the homotopy $(x+iy,t)
  \mapsto x+iy + t i \omega$ for $t \in [0,1]$ and $\omega \in
  Amp(X)$.  It follows that $\overline{\kl}_{0,>2}(X)$ is
  contractible.

  Hence, the restriction of the covering space $\pi:
  Stab^\dagger(X) \ra \kp_0^+(X)$ to $q(\overline{\kl}_{0,>2}(X))
  \subset \kp_0^+(X)$ is trivial (cf. Remark \ref{q_v_remark} for
  the definition of $q=q_{v_0}$) and there is a unique section $s$
  extending $\sigma \circ \theta: q(\kl(X)) \ra V(X)$ to
  $q(\overline{\kl}_{0,>2}(X))$. We now set $\sigma^* = s \circ q$.
\end{proof}

Next, we recall Bridgeland's description of the boundary of $U(X)$.
As a well known consequence we get a covering of the stability manifold by certain
translates of the closure $\overline{U}(X)$.

\begin{thm}{\cite[Thm.\ 12.1.]{BridgelandK3}}\label{U_boundary}
  The boundary $\del U(X)$ is contained in a locally finite union of real codimension-one 
  submanifolds.
  If $x \in \del U(X)$ is a general boundary point, i.e. lies only on one of these submanifolds, 
  then precisely one of the following possibilities hold.
  \begin{itemize}
  \item[$(A^{+})$] There is a rank $r$ spherical vector bundle $A$ such that
    the stable factors of the objects $\ko_x, x \in X$ are $A$ and $T_A(\ko_x)$.\footnote{ 
        Here $T_A$ denotes the spherical twist functor, cf.\ \cite[Sec.\ 8.1.]{HuybrechtsFM}. }
  \item[$(A^{-})$] There is a rank $r$ spherical vector bundle $A$ such that
    the stable factors of the objects $\ko_x, x \in X$ are $A[2]$ and $T_A^{-1}(\ko_x)$.
  \item[$(C_k)$] There is a non-singular rational curve $C$ and an integer $k$ such
    that $\ko_x$ is stable if and only if $x \notin C$.
    If $x \in C$, then it has a stable factor $\ko_C(k)[1]$.
  \end{itemize}
\end{thm}

\begin{rem}\label{U_boundary_rem}
  If $\sigma \in \overline{U}(X)$ satisfies condition $(A^+)$ or
  $(A^-)$, then the central-charges of $\ko_x$ and $A$ are co-linear:
  $Z(A)/Z(v_0) \in \IR_{>0}$.  This is precisely the condition we used
  in Definition \ref{wall_notation} to define the A-type wall
  $W_A(v(A),v_0) \subset \pd(X)$.
  Therefore the image $\overline{\pi}(\sigma)$ of in $\pd(X)$ lies on $W_A(v(A),v_0)$.

  Similarly, if $\sigma \in \overline{U}(X)$ satisfies condition $(C_k)$, then
  $\overline{\pi}(\sigma)$ lies on the wall $W_C(v(\ko_C),v_0)$ of type $C$.
\end{rem}

\begin{defn}
  Let $\WT(X) \subset Aut(\cd^b(X))$ be the group generated by the
  spherical twists $T_A^2,T_{\ko_C(k)}$ for all $(-2)$-curves $C$, $k
  \in \IZ$ and spherical vector bundles $A$, which occur in the
  description of the boundary $\del U(X)$ given in Theorem
  \ref{U_boundary}.
\end{defn}

\begin{rem}
  One can check, that all equivalences $\Phi \in \WT(X)$ have the
  property $\Phi^H(v_0)=v_0$.  This means we get a map 
  \[\WT(X) \lra  W^0(N(X),v_0) \subset O(N(X)),  \quad \Phi \mapsto \Phi^H.\]  
  As we will see in Proposition \ref{boundary_lemma}, this map is surjective.
\end{rem}

\begin{rem}\label{WT_Dagger}
  All equivalences $\Phi \in \WT(X)$ respect the distinguished component $Stab^\dagger(X)$.
  Indeed, the spherical twists $T_{\ko_C(k)}$ and $T_A^2$, which generate $\WT(X)$,
  map the corresponding boundary components of $\del U(X)$ into $\del U(X)$, 
  cf.\ {\cite[Thm.\ 12.1.]{BridgelandK3}}.
  We will study equivalences with this property more closely in section \ref{sec:aut_dagger}.
\end{rem}

The following lemma is an easy consequence of the proof of \cite[Prop.\ 13.2.]{BridgelandK3}.
\begin{lem}\label{cover_stab}
  The translates of the closed subset $\overline{U}(X)$ under the group $\WT(X)$ cover $Stab^\dagger(X)$:
  \[    \bigcup_{\Phi \in \WT(X)}  \Phi_* \overline{U}(X) =
  Stab^\dagger(X). \]
\end{lem}
One can show, moreover, that the intersections of the interiors $U(X)\cap \Phi_*
U(X)$ are empty unless $\Phi=id$. However, we will not need this refinement.


\section{Cusps and hearts of stability conditions}\label{sec:Degenerations}
In this section we proof our main geometric results about cusps and 
stability conditions.
We will introduce the notion of a linear degeneration to a cusp
in the K\"ahler moduli space and classify all paths in the stability
manifold mapping to linear degenerations. 
Moreover, we construct paths in the stability manifold
with special limiting hearts.

We will use the results proved in appendix \ref{sec:aut_dagger}.

\subsection{Linear degenerations in $\KM(X)$}
\begin{defn}
  Let $\gamma(t) \in KM(X), t \gg 0 $ be a path in the K\"ahler moduli 
  space and $[v]$ a standard cusp of $\KM(X)$.

  We say $\gamma(t)$ is a {\em linear degeneration} to a cusp $[v] \in \overline{KM}(X)$
  if there exists a lift $\alpha(t)$ of $\gamma(t)$ to $\pd^+(X)$ 
  and a vector $w \in \Gamma_X \cdot v$ such that
  \[ \alpha(t) = exp_w(x_0 + i\, t \, y_0 ) \]
  for some   $x_0  \in A(w)_\IR, \; y_0 \in C(L(w))$.
\end{defn}

\begin{prop}\label{lin_deg_prop}
  Let $\gamma(t)$ be a linear degeneration to $[v] \in \KM(X)$. 
  \begin{enumerate}
  \item The limit of $\gamma(t)$ in $\KM(X)$ is  
   \[ \limt \gamma(t)=[v] \in \KM(X). \]
  \item If $\beta(t)$ is another lift of $\gamma(t)$ to
    $\pd^+(X)$, then there is a $g \in \Gamma_X$ such that 
    \[ \beta(t) = exp_{w'}(x'_0 + i \, t\,  y'_0) \]
    for $w'=g \cdot w,\; x'_0=g \cdot x_0,\; y'_0 = g\cdot y_0$ and $t\gg 0$.
  \end{enumerate}
\end{prop}

\begin{proof}
  In \cite[2.2]{Looijenga} Looijenga constructs a basis of
  neighborhoods of $[v] \in KM(X)$ as follows.  Let $\Gamma_v = \{ g
  \in \Gamma^+_X \,|\, g \cdot v = v \}$. Consider the exponential
  parametrization
  \[ exp_v: A(v)_\IR \times C^+(L(v)) \lra \pd^+(X). \]
  The semi-group $L(v)_\IR \times C^+(L(v))$ acts on $A(v)_\IR \times C^+(L(v))$
  by translation, and hence also on $\pd^+(X)$. 
  For an open subset $K$ of $\pd^+(X)$ let 
  \[  U(K,v) = \Gamma_v \cdot (L(v)_\IR \times C^+(L(v))) \cdot K \subset \pd^+(X), \]
  which is also an open subset of $\pd^+(X)$.

  Then the images of $U(K,v)$ in $KM(X)=\Gamma_X^+ \setminus \pd^+(X)$,
  where $K \subset \pd^+(X)$ runs through all open and non-empty subsets of $\pd^+(X)$,
  form a basis of neighborhoods of $[v] \in \overline{KM}(X)$.

  It is easy to see that every linear degeneration $\gamma(t) \in KM(X)$ 
  lies eventually in any of the subsets $U(K,v)$.
  This shows the first claim.

  We now proceed to the second claim. By assumption, there are $g(t) \in
  \Gamma^+_X$ such that $ \beta(t) = g(t) \cdot \alpha(t)$. We have to
  show that it is possible to choose $g(t)$ independent of $t$ for $t
  \gg 0$. Note that, the action of $\Gamma_X$ on $\pd(X)$ is not fixed
  point free, and hence the elements $g(t)$ itself may depend
  on $t$.

  By \cite[Thm.\ 4.9.\ iv)]{BailyBorel} there is a subbasis $\{ U(K',v) \}$ of neighborhoods 
  of the cusp $[v]$ such that $\Gamma_v \setminus U(K',v)$ injects into $KM(X)$.
  
  Note that $\Gamma_v$ acts linearly on the tube model,
  i.e. $g \cdot \exp(x + iy) = exp( g \cdot x + i g \cdot y)$
  for the canonical action of $\Gamma_v$ on $A(v)$ and $L(v)$.
  Therefore the fixed point locus of an element $g \in \Gamma_v$ 
  is either disjoint form $\alpha(t)$ or contains $\alpha(t)$ for all $t$.

  This can be used to construct a uniform $g=g(t)$.
  Indeed, let $t_0 \gg 0$ and suppose, without loss of generality, that $\alpha(t_0)=\beta(t_0)$. 
  For $g \in \Gamma_v$ the closed sets
  \[ C(g) = \Set{ t \in [0,1] }{ g \cdot \alpha(t_0 + t) = \beta(t_0 + t) }  \]
  cover the interval $[0,1]$. Our argument above shows that 
  if $C(g) \cap C(g') \neq \emptyset$, then $C(g)=C(g')$. 
  By the properness of the action we see that $C(g) \neq \emptyset$ only for finitely many $g$.
  But an interval cannot be covered non-trivially by finitely many disjoint closed subsets.
  If follows that $C(\id)=\IR_{\geq t_0}$.
\end{proof}

\subsection{Linear degenerations of stability conditions}
We have the following natural maps from the stability manifold to the
K\"ahler moduli space.
\[
  \xymatrix{
    Stab^\dagger(X) \ar[d]^{\pi} \ar[dr]^{\bar{\pi}}  \ar[drr]^{\tilde{\pi}} & &  \\
    \kp_0^+(X) \ar[r]^{\pl} & \pd_0^+(X) \ar[r] &  \Gamma_X^+ \setminus \pd^+_0(X)=KM_0(X)
  }
\]
The goal of this section is to proof the following Theorem.
\begin{thm}\label{degeneration_thm}
  Let $[v] \in \KM(X)$ be a standard cusp and $ \sigma(t) \in
  Stab^\dagger(X)$ be a path in the stability manifold with the
  property that $\tilde{\pi}(\sigma(t)) \in \KM(X)$ is a linear
  degeneration to $[v]$.  Let $Y$ be the K3 surface
  associated to $[v]$ by Ma's theorem \ref{Ma_Thm}.  Then there
  exist
  \begin{enumerate}
  \item a derived equivalence $\Phi: \cd^b(Y) \xlra{\sim} \cd^b(X)$
  \item classes $x \in NS(Y)_\IR$, $y \in \overline{Amp}(Y)$ and
  \item a path $g(t) \in \GLT$
  \end{enumerate}
  such that 
  \[ \sigma(t)  =  \Phi_* (\sigma_Y^*(x,t\, y) \cdot g(t))  \]
  for all $t \gg 0$.\footnote{
    The stability condition $\sigma_Y^*(x,y)$ was
    constructed in Lemma \ref{sigma_extension_lemma}.
    If $y \in Amp(X)$, then $\sigma_Y^*(x,y )$ 
    agrees with Bridgeland's stability condition
    $\sigma_Y(x,y)$ (cf.\ Definition \ref{Existence}.)
  }

  Moreover, the hearts of $\sigma(t) \cdot g(t)^{-1}$ are independent
  of $t$ for $t \gg 0$.
  If $y \in Amp(X)$, then the heart can be explicitly described
  as the tilt $\ka_Y(x,y)$ of $Coh(Y)$.
\end{thm}

\begin{proof}
  By Proposition \ref{lin_deg_prop}, $\tilde{\pi}(\sigma(t))$ is a
  linear degeneration if and only if
  \[ \bar{\pi}(\sigma(t)) = exp_w(x_0 + i\, t \, y_0 ) \in \pd_0^+(X) \]
  for some $w \in \Gamma_X \cdot v$, $x_0  \in A(w)_\IR$, $y_0 \in C(L(w))$.

  Note that, $w$ is a standard vector and the Hodge structure on
  $w^\perp / w$ induced by $\HT(X,\IZ)$ is isomorphic to $v^\perp/v$.
  So, by Theorem
  \ref{reduction_to_lvl} and Remark \ref{fm_criterion} there is a
  derived equivalence $\Phi: \cd^b(X) \ra \cd^b(Y)$ such that
  $\Phi^H(w)=v_0$, and $\Phi_* \sigma(t) \in Stab^\dagger(Y)$.
  Hence we may assume, without loss of generality, that $w=v_0$.

  Now we claim, that there is a continuous path $g(t) \in \GLT$ such that 
  \[ \pi(\sigma(t) \cdot g(t)) = Exp_{v_0}(x_0 + i \, t \, y_0 ) \in
  \kp_0^+(X). \] Moreover, two such paths $g(t),g'(t)$ differ by an even
  shift $\Sigma_{2k}, k \in \IZ$ (cf.\ Example \ref{shift_example}),
  i.e. $g'(t) = \Sigma_{2k} \circ g(t)$.
  
  Indeed, in Remark \ref{q_v_remark} we constructed a section $q_{v_0}:\pd(X) \ra
  \kp(X)$ of the $\GL$-action on $\kp(X)$. 
  Hence, there is a unique $h(t) \in \GL$ such that
  $\pi(\sigma(t)) \cdot h(t) = q_{v_0}(exp_{v_0}(x_0 +i t y_0 ) ) =
  Exp_{v_0}(x_0 +i t y_0)$.  Every choice of a continuous lift $g(t)$ of
  $h(t)$ to $\GLT$ has the required property.  As $\GLT \ra \GL$ is a
  Galois cover with Galois group $\IZ$ acting by even shifts $k \mapsto
  \Sigma_{2k}$ the latter statement follows.

  We choose a $t_0>0$ such that $(t_0 y_0)^2>2$. 
  Lemma \ref{cover_stab} shows, that there is an auto-equivalence
  $\Psi \in \WT(X) \subset Aut^\dagger(\cd^b(X))$, such that
  $\Psi_* \sigma(t_0) \in \overline{U}(X)$.
  It is easy to see, that  there is a (unique) $k \in \IZ$ 
  such that $\Psi_* \sigma(t_0) \cdot \Sigma_{2k} \in \overline{V}(X)$.
  This allows us to assume, without loss of generality, that $\sigma(t_0) \in \overline{V}(X)$.

  By assumption on $t_0$, we have furthermore $\sigma(t_0)\in \overline{V}_{>2}(X)$.
  Now Lemma \ref{sigma_extension_lemma} shows that
  \[ \sigma(t_0) = \sigma^*_X(x, t_0 y). \] 
  We claim that the same holds for all $t \geq t_0$.
  Indeed, let $\sigma'(t)=\sigma_X^*(x, t y),\, t \geq t_0$. Then
  $\sigma(t)$ and $\sigma'(t)$ are
  two lifts of the path $Exp_{v_0}(x + ity) \in \kp_0^+(X)$ to $Stab^\dagger(X)$
  with the same value at $t=t_0$.  As $\pi$ is a covering-space we have
  $\sigma(t)=\sigma'(t)$.  This shows the claim and therefore the first part of
  the proposition.

  It remains to show, that the hearts $ \sigma^*_X(x, t y)$ are independent of
  $t \geq t_0$.  In the case $y \in Amp(X)$ this follows directly from Remark
  \ref{hearts}. The general case is more involved:

  We introduce the symbol  $\ka(t)$ for the heart, and  $\kp_t(\phi)$ for the slicing,
  of the stability stability condition $\sigma(t)$.
  Let $E \in \ka(t_0)$, we have to show, that $E \in \ka(t)$ for all $t\geq t_0$.
  We claim the following statements:
  \begin{enumerate}
  \item If $E \in \kp_t((0,1))$ for one $t \geq t_0$, then $E \in \kp_\tau([0,1])$ for all $ \tau \geq t_0$.
  \item If $E \in \kp_{t}(0)$ for one $t \geq t_0$, then  $E \in \kp_{\tau}(0)$ for all $\tau \geq t_0$.
  \item If $E \in \kp_{t}(1)$ for one $t \geq t_0$, then  $E \in \kp_{\tau}(1)$ for all $\tau \geq t_0$.
  \end{enumerate}
  Once we have shown the claim, we argue as follows. 
  Grouping Harder--Narasimhan factors in $\sigma(1)$ we get an exact triangle\footnote{
  We use the symbol $\xymatrix{A \ar@{.>}[r] & B}$ for a morphism $A \lra B[1]$.}
  \[ \xymatrix{ A \ar[r] &  E \ar[r] & B \ar@{.>}[r] & A } \] 
  with $A \in \kp_{t_0}(1)$ and $B \in \kp_{t_0}((0,1))$.
  Let now $t\geq {t_0}$. By $(1)$ we have $B \in \kp_t([0,1])$. 
  Taking Hader--Narasimhan filtration in $\sigma(t)$, yields exact triangles
  \[ \xymatrix{ 
    0=B_0 \ar[rr] & & B_1 \ar[rr]\ar[dl] & & B_2 \ar[rr]\ar[dl] & & B_3 = B\ar[dl]^{\alpha} \\
    & C_1 \ar@{.>}[ul] & & C_2 \ar@{.>}[ul] && C_3 \ar@{.>}[ul] 
  }\]
  with $C_1 \in \kp_t(1),C_2 \in \kp_t((0,1)),C_3 \in \kp_t(0)$. By (2) we have
  $C_3 \in \kp_{t_0}(0)$, but $B \in \kp_{t_0}((0,1))$, therefore $(\alpha:B \ra C_3)=0$ 
  which is only possible if $C_3=0$.
  This means that $B \in \kp_t((0,1])=\ka(t)$. 
  By (3) we also have $A \in \kp_t(1) \subset \ka(t)$, which implies $E \in \ka(t)$
  since $\ka(t)$ is extension closed.

  Ad 1: 
  Let $\omega \in Amp(X)$ be an ample class, then also $y + s \omega$ is
  ample for all $s>0$. For $s \geq 0, t \geq {t_0}$, 
  let $\sigma(t,s)=\sigma^*(x,t(y + s \omega)) \in \overline{V}(X)$ 
  (cf. Lemma \ref{sigma_extension_lemma}).
  Note that $\sigma(t,0)=\sigma(t)$.

  The property $E \in \kp_\sigma((0,1))$ is clearly open in $\sigma$.
  Hence we find $\eps>0$, with $E \in \kp_{t,s}((0,1))$ for all $\eps > s \geq 0$.
  If $s>0$, then $\sigma(t,s) \in V(X)$ and the heart $\ka(\sigma(t,s))$ is independent 
  of $t \geq {t_0}$ (cf. Remark \ref{hearts}). 
  Therefore, $E \in \ka(\sigma(\tau,s))$ for all $\tau \geq {t_0},s>0$.
  Taking the limit $s \ra 0$ we find $E \in \kp_\tau([0,1])$, for all $\tau \geq {t_0}$.

  Ad 3: Applying a shift we reduce this statement to (2).

  Ad 2: The property $E \in \kp_\tau(0)$ is clearly closed in $\tau$. It suffices to show openness.
  Fix $t\geq {t_0}$ with  $E \in \kp_t(0)$ and $1 > \eps > 0$.

  Let $T$ be the set of objects which occur as semi-stable factors of $E$ in
  a stability condition $\sigma(\tau), |t-\tau| \leq \eps$. Then the set $T$ has bounded mass 
  (cf.\ proof of \cite[Prop. 9.3.]{BridgelandK3}), and therefore the set of Mukai vectors 
  $S=\Set{v(A)}{ A \in T}$ is finite (cf. \cite[Lem.\ 9.3]{BridgelandK3}).

  Writing out formula for $Z_t(v)$ as in \cite[Sec.\ 6]{BridgelandK3}, we see that
  $Im(Z_t(v))$ vanishes  if and only if $Im(Z_\tau(v))$ vanishes for all $\tau \geq {t_0}$.
  It follows that $S$ decomposes as a disjoint union $S = S^0 \dunion S'$,
  where $Im(Z_\tau(v))=0$ (or $\neq 0$) for all $\tau \geq {t_0}$,  
  if $v \in S^0$ (or $v \in S'$ respectively).

  As the interval $[t-\eps,t+\eps]$ is compact and $S$ is finite, 
  there exists a $1 > \alpha > 0$ such that $|arg(Z_\tau(v))| > \alpha$ for 
  all $v \in S', |\tau - t| \leq \eps$.

  Making $\eps$ again smaller, we can assume that 
  \[ E \in \kp_{\tau}((-\alpha,\alpha)) \forall |t - \tau| \leq \eps. \tag{\#}\]
  It follows that all semi-stable factors $A$ of $E$ in stability condition 
  $\sigma(\tau)$ with $ |\tau - t| \leq \eps$, have the property that $v(A) \in S^0$.
  Moreover, as $A \in \kp_{\tau}((-\alpha,\alpha))$ and $arg(Z_\tau(A))\in
  \IZ$, we find $A \in \kp_\tau(0)$ and therefore $E \in \kp_\tau(0)$.
\end{proof}

\subsection{Limiting hearts}\label{sec:limiting-hearts}

\begin{defn}
  Let $\kc$ be a category. For a sequence of full subcategories $\ka(t) \subset \kc, t \gg 0$ 
  define the {\em limit}
  to be the full subcategory of $\kc$ with objects
  \[ \underset{t \ra \infty}{\lim} \ka(t)  =\Set{ E \in \kc }{E \in \ka(t) \forall t \gg 0 }. \]
\end{defn}

\begin{thm}\label{limiting_hearts}
  Let $[v] \in \KM(X)$ be a standard cusp, and $Y$ the associated K3 surface.
  Then, there exist a path $\sigma(t) \in Stab^\dagger(X), t \gg 0$ 
  and an equivalence $\Phi: \cd^b(Y) \sra \cd^b(X)$ such that 
  \begin{enumerate}
  \item $\limt \tilde{\pi}(\sigma(t)) = [v] \in \KM(X)$ and
  \item $\limt \ka(\sigma(t)) = \Phi( Coh(Y) ) $
  \end{enumerate}
  as subcategories of $\cd^b(X)$.
\end{thm}

\begin{proof}
  By Theorem \ref{reduction_to_lvl} there is a derived equivalence
  $\Phi: \cd^b(X) \ra \cd^b(Y)$ mapping $[v]$ to $[v_0]$ and 
  $\sigma(t)$ into the distinguished component $Stab^\dagger(Y)$.
  Hence we may assume, without loss of generality, that $[v]=[v_0]$ and $X=Y$.

  Let $\omega \in Amp(X)$ and consider the 
  sequence $\sigma_{X}(\omega t, \omega t) \in Stab^\dagger(X)$.

  As in the proof of Proposition \ref{lin_deg_prop}, we see that
  $\tilde{\pi}(\sigma(t))=[exp(t \beta + i t \omega)]$ converges to
  $[v_0]$. Indeed, if the vector $\beta + i t \omega$ lies in a
  principal open $U(K,v_0)$, then also $t \beta + i t \omega \in
  U(K,v_0)$, since $U(K,v)$ is invariant under the additive action of
  $L(v_0)_\IR$ on $T(N,v_0)$.

  The heart $\ka_{X}(t \omega, t \omega)$ consists of objects $E \in \cd^b(X)$
  with $H^0(E) \in \kt(t)$, $H^{-1}(E) \in \kf(t)$ and $H^i(E)=0$ for all $i \notin \{0,-1\}$,
  where
  \begin{align*}
    \kt(t)&=\Set{ A \in Coh(X)}{A\text{ torsion or } \mu_{\omega}^{min}(A/A_{tors}) > t \omega^2 }\\
    \kf(t)&=\Set{ A \in Coh(X)}{A\text{ torsion free and } 
      \mu_{\omega}^{max}(A) \leq t \omega^2 }.
  \end{align*}
  As $\omega^2 >0$ every sheaf $A$ lies in $\kf(t)$ for $t$ sufficiently large.
  Similarly no sheaf $A$ lies in $\kt(t)$ for all $t \gg 0$.
  Thus we find  $\underset{t \ra \infty}{\lim} \ka_{X}(\omega t, \omega t) \isom Coh(X)[1].$
\end{proof}


\subsection{Metric aspects}

In this section we will define a natural Riemannian metric on the period domain
$\pd(N)$ and show that linear degenerations are geodesics.

Let $N$ be a lattice of signature $(2,\rho)$.  The natural action of
the real Lie group $G=O(N_\IR)$ on $\IP(N_\IC)$ induces a transitive
action of $G$ on $\pd(N)$.  Let $[z] \in \pd(N)$ be a point and let $P
\subset N_\IR$ be the positive definite subspace spanned by $Re(z)$
and $Im(z)$. The stabilizer of $[z] \in \pd(N)$ is the compact subgroup
\[ K_{P} = \Set{ g \in G }{ g \cdot [z] = [z] } \isom SO(P) \times
O(P^\perp). \] Let $\fk_{P} \subset \fg$ be the Lie algebra of $K_P
\subset G$.  We can identify the tangent space $T_{[z]}\pd(N)$ with
the quotient $ \fg / \fk_{P}. $

As $G$ is semi-simple, the Killing form $B$ on $\fg$ is non-degenerate.
Let $\fm_P=\fk_P^\perp$ be the orthogonal complement of $\fk_P$ with
respect to $B$.  More explicitly, by
\cite[III.B.ii]{Helgason} we have $ B(X,Y) = \rho \cdot Tr(X \circ Y)$ and
\[ \fm_P = \Set{ X \in \fg }{ X(P) \subset P^\perp, \; X(P^\perp)
  \subset P }. \] We get a Cartan decomposition $\fg = \fk_P \vsum
\fm_P$.  The restriction of $B$ to $\fm_P$ is positive definite and
induces an invariant Riemannian metric on $\pd(N)$ via the canonical
isomorphism $ \fm_P \isom T_{[z]} \pd(N)$
(cf. \cite[III.7.7.4]{Helgason}).

Now \cite[IV.3, Thm.\ 3.3.iii]{Helgason} shows, that the geodesics 
of $\pd(N)$ through $[z]$ are given by the images $exp(t X) \cdot [z]$
of the one-parameter sub-groups $\{ exp(t X) \,|\, t\in \IR \} \subset G$ with $X \in \fm_P$.
We will construct a special $X \in \fm_P$ such that $exp(t X) \cdot [z]$
is a linear degeneration through $[z]$.

Let $v_0 \in N$ be a standard vector and let $x+iy \in T(N,v_0)$ with
$[z] = exp_{v_0}(x+iy)$.  Recall that $x \in N_\IR / v_0 \IR$ with
$x.v_0=-1$. There is a unique lift $x_0$ of $x$ to $N_\IR$ such that
$(x_0)^2=0$.  Indeed, if $\tilde{x}$ is any lift, then
$x_0=\tilde{x}-\half (\tilde{x}^2) v_0$ 
has the required property. 

Set $x_1=-x_0$ and let $U \subset N_\IR$ be the hyperbolic plane
spanned by $(v_0,x_1)$. 
We get a one-dimensional Lie sub-algebra
\[ \fa(v_0,x) = \mathfrak{so}(U) \subset \fg \]
which depends on the choice of $v_0$ and $x$.

\begin{lem}
  The Lie algebra $\fa(v_0,x)$ is contained in $\fm_P.$
\end{lem}
\begin{proof}
Let $R = U^\perp$ and decompose $N_\IR$ as a direct sum $N_\IR=\<v_0\> \vsum \<x_1\> \vsum R$.
We write elements of $N_\IR$ as column vectors $(a,b,c)^{tr}=a v_0 + b v_1
+ c$ with $a,b \in \IR$ and $c \in R$.
We have
\[ Exp_{v_0}(x+iy) = (-\half y^2,-1,iy)^{tr}. \]
The two-plane $P$ is spanned by the vectors
$(-\half y^2,-1,0)^{tr}$ and $(0,0,y)^{tr}$.
The Lie algebra $\fa(v_0,a)$ consists of all matrices
\[ A_\lambda = \begin{pmatrix}
  \lambda & 0 & 0 \\
    0 & -\lambda & 0 \\
    0 & 0 & 0
  \end{pmatrix}, \quad \lambda \in \IR.  \]
One checks easily, that $A_\lambda(P) \perp P$.
The orthogonal complement of $P$ consists of all vectors
$\gamma=(\half b y^2 ,b,c)^{tr}$ with $c.y=0$. Therefore,
\[ A_\lambda(\gamma)=(\lambda \half b y^2, - \lambda b ,
0)^{tr}=\lambda b \alpha \in P. \]
This shows that $\fa(v_0,x) \subset \fm_P$.
\end{proof}

\begin{lem}
  Let $[z] = exp_{v_0}(x+iy) \in \pd(N)$ and $A_\lambda \in \fa(v_0,x)$, then the action of
  $exp(A_\lambda)$ is given by 
  \[ exp(A_\lambda) \cdot [z] = exp_{v_0}( x + i t y ), \]
  where $t = exp(\lambda)$.
\end{lem}
\begin{proof}
  As above we write elements of $N_\IR$ as column vectors with
  respect to the  decomposition $N_\IR=\<v_0\> \vsum \<x_1\> \vsum
  R$. Similarly, endomorphisms are represented by matrices. We have  
  \[ exp(A_\lambda) = \begin{pmatrix}
    t & 0 & 0 \\
    0 & t^{-1} & 0 \\
    0 & 0 & id_R
  \end{pmatrix},   \]
  where $t = exp(\lambda)$. Therefore,
  \begin{align*}
    exp(A_\lambda) \cdot P &= \< ( - t \half y^2,-t^{-1},0)^{tr}, ( 0 , 0 ,y)^{tr}\> \\
    &= \< ( -  \half (t y)^2, -1 ,0)^{tr}, (0,0,t y)^{tr}\>,
  \end{align*}
  which is the two-plane spanned by the real- and imaginary parts of
  the vector $Exp_{v_0}(x + ity)$.
\end{proof}

\begin{cor}\label{cor:geodesics}
  For all $x + iy \in T(N,v_0)$, the path \[ \alpha(t) = exp_{v_0}(x + i\, exp(t)\, y)
  \in \pd(N)\] is a  geodesic of constant speed.
\end{cor}

If $\Gamma \subset G$ is a discrete subgroup acting properly and
discontinuously on $\pd(N)$ then the quotient $\Gamma \setminus \pd(N)$
inherits a Riemannian metric on the smooth part $(\Gamma \setminus \pd(N))_{reg}$.
Geodesics in $(\Gamma \setminus \pd(N))_{reg}$ are locally the images
of geodesics on $\pd(N)$. More generally we define geodesics in
$\Gamma \setminus \pd(N)$ to be the images of geodesics in $\pd(N)$.

This discussion applies in particular to the K\"ahler moduli space
of a K3 surface $X$. From the definition of linear degeneration and
Corollary \ref{cor:geodesics} we get immediately the following statement.

\begin{cor}\label{linear_geodesics}
  Linear degenerations are geodesics in the K\"ahler moduli space $KM(X)$.
\end{cor}
Note however, that our parametrization $exp(x+ity)$ is not of constant
speed.

We conjecture the following converse to the above corollary.
\begin{conj}\label{geodesic_conjecture}
  Let $[v] \in \KM(X)$ be a zero-dimensional cusp. Then every geodesic
  converging to $[v]$ is a linear degeneration.
\end{conj}
We have the following evidence. The conjecture holds true in the case
$X$ has Picard rank one. Then, $\pd(N)^+$ is isomorphic to the upper
half plane and the geodesics converging to the cusp $i \infty$ are
precisely the vertical lines, which are our linear degenerations.

If one uses the reductive Borel--Serre compactification $\KM(X)^{BS}$ to compactify $KM(X)$, then
the analogues conjecture seems to follow from 
\cite{JiMacPherson}.
 Indeed, Ji and MacPherson describe the boundary
of $\KM(X)^{BS}$ as a set of equivalence classes of, so called,
EDM-geodesics (cf. \cite[Prop.\ 14.16]{JiMacPherson}).  Moreover, all
EDM-geodesics are classified in \cite[Thm.\ 10.18]{JiMacPherson}.
They are of the form $(u,z,exp(tH)) \in N_Q \times X_Q \times A_Q$
where $Q \subset G$ is a rational parabolic subgroup and $N_Q \times
X_Q \times A_Q \isom \pd(N)$ is the associated horocycle
decomposition. We think, that linear degenerations to $[v]$ are
the geodesics associated to the stabilizer group $G_{[v]}$ of $[v] \in
\IP(N_\IC)$. Moreover, all geodesics $\gamma$ that converge to the
boundary component $e([v]) \subset \KM(X)^{BS}$ associated to $G_{[v]}$ should have the
EDM property.  It follows form the classification, that $\gamma$ is of the form
$(u,z,exp(tH))$ for some rational parabolic subgroup $Q \subset G$. 
Since $\gamma$ converges to $e([v])$, we have $Q=G_{[v]}$ and
therefore $\gamma$ should be a linear degeneration.

There is a natural map $\KM(X)^{BS} \ra \KM(X)$ (cf. \cite[III.15.4.2]{BorelJi}).
One should be able to prove the full conjecture by studying the fibers
of this map over a cusp $[v] \in \KM(X)$.


\section{Moduli spaces of complexes on K3 surfaces}\label{sec:Moduli}
In this section we construct K3 surfaces as moduli spaces of stable
objects in the derived category of another K3 surface.  First we
introduce a moduli functor, which is a set-valued version of
Lieblich's moduli stack cf. \cite{Lieblich}.  We will show in
subsection \ref{sec:fm_moduli}, that Fourier--Mukai equivalences
induce natural isomorphisms between moduli spaces.  
Finally, in subsection \ref{sec:Reconstruction} we prove our main theorem.

Before we can give the actual definition, we recall the notion of a perfect
complex in the first subsection. Moreover, we establish a base-change formula
and a semi-continuity result which will be important later.

\subsection{Perfect complexes}
We denote by $\kd(X)$ the unbounded derived category of coherent sheaves on $X$.

\begin{defn}
  Let $X \ra T$ be a morphism of schemes.  A complex $E \in \kd(X)$ is called
  {\em relatively $T$-perfect}, if there is an open cover $\{U_\nu\}$ of $X$
  such that $E|_{U_\nu}$ is quasi-isomorphic to a bounded complex of $T$-flat
  sheaves of finite presentation.
  
  We call $E$ {\em strictly $T$-perfect} if $E$ itself is quasi-isomorphic to a
  bounded complex of $T$-flat sheaves of finite presentation.
\end{defn}

\begin{lem} \cite[Cor.\ 2.1.7]{Lieblich} \label{perfect_lemma}
  If $T$ is an affine scheme and $f:X \ra T$ is a flat, finitely presented and 
  quasi-projective morphism, then every relatively $T$-perfect complex is strictly $T$-perfect.
\end{lem}

The following base-change result is presumably well known to the
experts. The main difference to the usual base change theorems like
\cite[Prop.\ 5.2]{Hartshorne} is that we do not assume flatness 
of any maps, but perfectness of the complex.

\begin{prop}[Base Change]\label{NonflatBC}
  Consider a  diagram of separated, noetherian schemes
  \[ \xymatrix{ X' \ar[r]^{j} \ar[d]^{q} & X \ar[d]^{p}  \\ 
    Y' \ar[r]^i \ar[d]^{v} & Y \ar[d]^{u} \\ 
    S'  \ar[r]^{k} & S } \]
  where $p$ is proper and both squares are Cartesian. 
  Let $E \in \kd^b(X)$ be a strictly $S$-perfect complex.
  Then the base change morphism
  \[ \DL i^* \DR p_* E \lra \DR q_* \DL j^* E\]
  is an isomorphism.

  The same holds true if $E \in \kd^b(X)$ is only $S$-perfect but $p$ is flat, 
  finitely presented and projective.
\end{prop}

\begin{proof}
As the statement is local in $Y'$ we may assume that $Y,Y',S,S'$ are affine.
If $E$ is $S$-perfect and $p$ is flat, finitely-presented and
projective, then Lemma \ref{perfect_lemma} shows that $E$ is strictly
$S$-perfect.  Hence it suffices to treat the case that $E$ is a
bounded complex of $S$-flat coherent sheaves on $X$.

Step 0) Choose a finite open affine cover $\gu=\{ U_\nu \}$ of $X$. 
The Cech-complex $\kc(\gu,E) \in \kd^b(X)$ of $E$ with respect to $\gu$ is the
total complex of the following double complex of quasi-coherent sheaves on $X$
\[ \kc^q(\gu,E^p)=\prod_{\nu_0 < \dots < \nu_q} \iota_* E^p|_{U_{\nu_0}\cap \dots \cap U_{\nu_q}}, 
\quad d_1^{pq}=d_E^p, d_2^{pq}=(-1)^p \delta^{q}, \]
where $\iota:U_{\nu_0}\cap \dots \cap U_{\nu_q} \ra X$ denotes the inclusion.
It comes with a canonical morphism
\begin{align}\label{Equis} E \lra \kc(\gu,E), \quad m \in E^p \mapsto (m|_{U_\nu})_\nu \in \kc^0(\gu,E^p).
\end{align}
which is a quasi-isomorphism. This can be checked using the spectral
sequence for double complexes and the vanishing of the $E_2^{pq}$ in
degrees $q\neq 0$.

Step 1) We claim that the sheaves $\kc^n(\gu,E), n\in \IZ$ are acyclic for $p_*$.
The sheaf $\kc^n(\gu,E)$ is a direct sum of sheaves of the form $\iota_* E^p|_{U'}$
where $U'=U_{\nu_0}\cap \dots \cap U_{\nu_n}$.
Since $X$ is separated $U'$ is affine. The morphism $p': U' \lra Y $ between affine schemes 
is affine and hence all higher direct images 
$\DR^i p' E^p|_{U'} = 0, i>0$ vanish. 
Hence
\[ \DR p_* E \isom p_* \kc(\gu,E). \]

The sheaves $\kc^n(\gu,E)$ are still $S$-flat, since $p_* E|_{U'}$ 
are given by restriction of scalars along the morphism of affine 
schemes $p':U' \ra Y$. This shows that
\begin{align}\label{pi-quis} \DL i^* \DR p_* E  \isom i^* p_* \kc(\gu,E).
\end{align}

Step 2) On the other hand we have
\[ \DL j^* E \isom j^* \kc(\gu,E), \]
since $\kc^n(\gu,E)$ is $S$-flat for all $n \in \IZ$.

We claim that $j^* \kc^n(\gu,E), n \in \IZ$ are acyclic for $q_*$.

Again we use that $j^* \kc^n(\gu,E)$ are direct sums of sheaves of 
the form $j^* \iota_* E^p|_{U'}$ with $U'=U_{\nu_0}\cap \dots \cap U_{\nu_q}$.
Consider the open affine subset $V'=j^{-1}(U') \subset X'$ and the following diagram
\[ \xymatrix { 
  V' \ar[d]^{\iota'} \ar[r]^{j'} \ar@/_2pc/[dd]_{q'} & U'\ar[d]^{\iota} \ar@/^2pc/[dd]^{p'} \\
  X' \ar[r]^{j} \ar[d]^q & X \ar[d]^p \\
  Y' \ar[r]^{i} & Y 
}
\] 
We have
\[ j^* \iota_* E^p|_{U'} = j^* \iota_*  \iota^* E^p  =  \iota'_* j'^* \iota^* E^p = \iota'_* \iota'^*  j^* E^p. \]
In the second step we use base change for open inclusions of affine schemes into separated schemes.
It follows that the higher direct images vanish:
\[ \DR^iq_*( j^* \iota_* E^p|_{U'} ) =\DR^iq_*(\iota'_* \iota'^*  j^* E^p)= \DR^i{q'}_* (\iota'^*  j^* E^p) = 0\]
for all $i>0$. We used in the second step that $\iota_*$ is exact and in the third step that $q'$ is affine.

This shows that
\begin{align}\label{qj-quis} \DR q_* \DL j^* E \isom q_* j^* \kc(\gu,E).
\end{align}

Step 3) The base change morphism 
\[ \theta: \DL i^* \DR p_* E  \lra  \DR q_* \DL j^* E  \]
can be constructed using the adjunction of $\DL q^*,\DR q_*,$ and $\DL p^*,\DR p_*$.
It can be computed on appropriate resolutions using the adjunction of functors of sheaves
between $q^*,q_*$ and $p^*,p_*$.

Under the quasi-isomorphisms (\ref{qj-quis}) and (\ref{pi-quis}) the morphism $\theta$ 
is given by a morphism of complexes whose components are base-change morphisms 
\[   q_*  j^* \iota_*  E^p|_{U'} = q'_* j'^* E^p|_{U'} \lra  i^* p'_*  E^p|_{U'} = i^* p_*  \iota_*  E^p|_{U'}  \]
for the affine schemes $V',U',Y,Y'$ and thus isomorphisms.
\end{proof}

\begin{prop}[Semi-continuity]\label{semicontinuity}
  Let $X \ra T$ be a proper morphism between 
  separated, noetherian schemes and let $E \in \kd^b(X)$ be a $T$-perfect complex.
  
  For $t \in T$, denote by $i_t:X_t= X \times_T \{t\} \ra X$ the inclusion of the fiber
  and by $E_t = \DL i_t^* E$ the derived restriction. 
  Then the function
  \[ \phi^i: T \ra \IZ: t \mapsto dim_{k(t)} \IH^i(X_t,E_t) \]
  is upper semi-continuous. 
\end{prop}
\begin{proof}
  We may assume $T$ is affine and $E$ is a bounded complex of $T$-flat coherent sheaves.
  By Theorem \ref{NonflatBC} we have
  \[ \IH^i(X,E_t) = \curly{H}^i(\DL i_t^* \DR p_* E ) \]
  as sheaves on $\{t\} = Spec(k(t))$.
  In the proof of this theorem we saw that $\DR p_* E$ can be represented by a 
  bounded complex of $T$-flat, quasi-coherent sheaves
  with coherent cohomology.

  The remaining arguments are identical to the proof of the semi-continuity theorem 
  for a $T$-flat sheaf $E$ in \cite[Thm.\ III.12.8]{Hartshorne}.
\end{proof}


\subsection{Moduli functor}
Let $X$ be a K3 surface and $T$ be a scheme over $\IC$.
For a point $t \in T(\IC)$ we denote by $i_t: X \ra X \times T$ the inclusion of the fiber
and for a complex $E \in \kd^b(X \times T)$ let $E_t = \DL i_t^* E$ be the restriction.
\begin{defn}
  For $v \in N(X)$ and $\sigma \in Stab(X)$ consider the moduli functor
  \[ \km^\sigma_X(v): (Shm/\IC)^{op} \ra (Set), 
  \quad T \mapsto \{ E \in \kd^b(X \times T) \,|\, (*) \, \}/ \sim. \]
  Here  $(Shm/\IC)$ is the category of separated schemes of finite type over $\IC$.\footnote{For 
    technical reasons related to Proposition \ref{NonflatBC} we have to restrict ourselves to this 
    subcategory. }
  The symbol $(*)$ stands for the following conditions.
  \begin{enumerate}
  \item The complex $E$ is relatively $T$-perfect.
  \item For all $t\in T(\IC)$ the restriction $E_t \in \kd^b(X)$ is $\sigma$-stable of Mukai vector $v(E_t)=v$.
  \end{enumerate}
  The equivalence relation $\sim$ is defined as follows.
  We have $E \sim E'$ if and only if 
  there is an open cover $\union T_\nu = T$ of $T$ such that for all $\nu$ 
  there is a line bundle $L \in Pic(T_\nu)$ and an even number $k \in 2 \IZ$ with 
$E \isom E' [k] \tensor pr_2^* L $
  on $X \times T_\nu$.

  To a morphism of schemes $i: S \ra T$ in $(Shm/\IC)$ the functor assigns the map
  \[ \DL i_X^*:  \km^\sigma_X(v)(T) \lra \km^\sigma_X(v)(S)\]
  sending $E \in \kd^b(X \times T)$ to  $\DL i_X^* E$, where $i_X=\id_X \times i $. 
\end{defn} 

\subsection{Moduli spaces under Fourier--Mukai transformations}\label{sec:fm_moduli}

\begin{thm}\label{moduli_under_fm}
  Let $\Phi: \kd^b(X) \ra \kd^b(Y)$ be a Fourier--Mukai equivalence 
  between two K3 surfaces $X$ and $Y$, then $\Phi$ induces an isomorphism of functors
  \[ \km^\sigma_X(v) \xlra{\isom} \km^{\Phi_*\sigma}_Y(\Phi_*^H v) \]
\end{thm}

\begin{proof}
  Denote by $p_T,q_T$ and $\pi$ the projections from $X \times Y \times T$ to $X \times T, Y \times T$ 
  and $X \times Y$, respectively.
  Let $\kp \in \kd^b(X \times Y)$ be the Fourier--Mukai kernel of $\Phi$.
  We claim that the map
  \begin{align}\label{MMorp}  
    \km^\sigma_X(v)(T) \ni E  \mapsto \Phi_T(E) := \DR q_{T *}( p_T^* E \tensor^\DL \pi^* \kp )
  \end{align}
  induces a natural transformation $\km^\sigma_X(v) \ra \km^{\Phi_*\sigma}_Y(\Phi_*^H v)$ 
  between the moduli functors.
  For this we need to check the following properties.
  \begin{enumerate}
  \item The complex $\Phi_T(E)$ is relatively $T$-perfect.
  \item Naturality: For all  $i: S \ra T \in (Shm/\IC)$ and $E \in \km^\sigma_X(v)(T)$ we have  
    \[ \DL i^*_Y \circ \Phi_T(E) = \Phi_S \circ \DL i^*_X(E).  \] 
  \item For all $t\in T(\IC)$ the complex $\Phi_T(E)_t$ is $\Phi_*\sigma$-stable of Mukai vector $\Phi^H(v)$.
  \end{enumerate}

  Ad 1) As $X \times Y$ is a smooth projective scheme, we can represent $\kp$ by a bounded complex of
  coherent, locally free sheaves.  Since locally-free sheaves are acyclic for $\_\tensor^\DL\_$ 
  we find that $\pi^* \kp \tensor^\DL p_T^* E$ is $T$-perfect.

  Now, \cite[SGA 6,III,4.8]{SGA6}\footnote{ Grothendieck and Illusie
    use a slightly different definition of relative perfectness. The
    definition agrees with ours in the case of flat morphisms of
    finite type between locally noetherian schemes (cf. \cite[Def.\
    2.1.1.]{Lieblich} ff). The projections $X \times T \ra T$ and $Y
    \times T \ra T$ clearly have this property.
  } shows that the pushforward $\DR q_{T *}( p_T^* E \tensor^\DL \pi^* \kp )$ 
  is still $T$-perfect. 

  Ad 2) This is a direct computation using the base-change formula Proposition \ref{NonflatBC}. 

  Ad 3) By (2) we have $ \Phi_T(E)_t \isom \Phi(E_t)$.   
  Now $(3)$ follows from the definition of $\Phi_* \sigma$ and $\Phi^H(v)$.

  Finally we need to show that (\ref{MMorp}) is an isomorphism.
  For this we use the following straight forward Lemma.
  \begin{lem}
    Let $\Phi: \kd^b(X) \ra \kd^b(Y)$ and $\Psi: \kd^b(Y) \ra \kd^b(Z)$ be derived equivalences between K3 surfaces $X,Y,Z$,
    then
    \[ \Phi_T \circ \Psi_T = (\Phi \circ \Psi)_T : \km^\sigma_X(v) \lra \km^{\tau}_Z(w)  \]
    where $w=\Psi^H(\Phi^H(v))$ and $\tau=\Psi_*(\Phi_*(\sigma))$.  \qedhere
  \end{lem}   
  An inverse to the equivalence $\Phi$ is given by a Fourier--Mukai 
  transformation $\Psi$ with kernel $\kp\dual[2]$.
  Moreover, the kernels of the compositions $ \Psi \circ \Phi$,  $\Phi \circ \Psi$
  are quasi-isomorphic to $\ko_\Delta \in \kd^b(X \times X)$ and
  $\ko_\Delta \in \kd^b(Y \times Y)$ respectively (cf.\ \cite[5.7, ff.]{HuybrechtsFM}).
  Clearly $\ko_\Delta$ induces the identity on $\km^\sigma_X(v),\km^\sigma_Y(v)$.
  This shows that $\Phi_T$ and  $\Psi_T$ are inverse natural transformations. \hfill\qed
\end{proof}


\subsection{More on stability conditions}
Before we can finally state our main result on moduli spaces of stable
objects we need another digression on stability conditions.
First, we prove a classification result for semi-stable objects,
then we introduce $v_0$-general stability conditions and derive some
basic properties.

\begin{prop}\label{v_0_semi_stable}
  Let $\sigma \in U(X)$ be a stability condition. Then
  an object $E$ is $\sigma$-semi-stable with Mukai vector $v_0=(0,0,1)$ 
  if and only if there is an $x \in X$ and $k \in 2\IZ$ such that $E\isom \ko_x[k]$.
  \[ \Set{ E \in \cd^b(X)}{ v(E)=v_0, E\; \sigma\text{-semi-stable} }
  = \Set{ \ko_x[2k] }{ x \in X, k\in \IZ }. \]
\end{prop}
\begin{proof}
  Let $\sigma \in U(X)$ be a stability condition. 
  The objects $\ko_x, x\in X$ are $\sigma$-stable by Theorem \ref{stab_criterion}
  and hence in particular semi-stable.

  Let $E'$ be a $\sigma$-semi-stable object with Mukai vector $v_0$. 
  Applying an element in $\GLT$ we can assume that $\sigma$ is of the form $\sigma(\omega,\beta)$. 
  There is a unique $k \in \IZ$ such that $E'[k]=E$ lies in the heart $\ka(\omega,\beta)$.
  As $Z_\sigma(E)=(-1)^k Z_\sigma(v_0)=-(-1)^k$ has to lie in $\IH \union \IR_{<0}$ the
  number $k$ has to be even and the phase of $E$ is one. 
  Take a Jordan--H\"older filtration
  \[ 0 \subset E_1 \subset E_2 \subset \cdots \subset E_n=E \]
  of $E$ in $\ka(\omega,\beta)$. The stable quotients $A_i =E_i / E_{i-1}$
  have the same phase as $E$. Hence we can use the classification result
  of Huybrechts, \cite[Prop.\ 2.2]{HuybrechtsDAeq2008}, which shows that
  $A_i = F[1]$ for a vector bundle $F$ or $A_i=\ko_x$ for some $x \in X$
  Note that the Mukai vectors in these two cases are given by
  \[ v(F[1])=-(r,l,s) \text{ with } r>0, \quad v(\ko_x)=(0,0,1). \]
  By assumption we have $\sum_i v(A_i)=v(E)=(0,0,1)$.
  Hence the sum over the ranks of all occurring vector bundles has to be zero.
  This is only possible if there are none of them.
  Hence $E$ is an extension of skyscraper sheaves. Comparing 
  Mukai vectors again, one sees that $E$ has to be of the form $\ko_x$ for some $x \in X$.
\end{proof}

\begin{defn}\label{def_v_general}
  Fix a Mukai vector $v\in N(X)$. A stability condition $\sigma \in Stab^\dagger(X)$ 
  is called {\em $v$-general} if every $\sigma$-semi-stable object $E$ of Mukai 
  vector $v(E)=v$ is $\sigma$-stable.
\end{defn}

\begin{lem}\label{v_0_general}
  Every stability condition $\sigma \in U(X)$ is $v_0=(0,0,1)$-general.
  No stability condition $\sigma \in \del U(X) \subset Stab^\dagger(X)$ is $v_0$-general.
\end{lem}

\begin{proof}[Proof of Lemma]
  By Proposition \ref{v_0_semi_stable} all $\sigma$-semi-stable objects of
  Mukai vector $v_0$ are shifts of skyscraper sheaves.
  All skyscraper sheaves $\ko_x$ are $\sigma$-stable by Proposition \ref{stab_criterion}.

  For the second claim note that $\ko_x$ remains semi-stable for $\sigma \in \overline{U}(X)$. 
  If $\sigma \in \overline{U}(X)$ and all $\ko_x$ are $\sigma$-stable,
  then $\sigma \in U(X)$ by Proposition \ref{stab_criterion}.
  Hence for $\sigma \in \del U(X)$ there are strictly semi-stable skyscraper sheaves. 
  This means $\sigma$ is not $v_0$-general.
\end{proof}

\begin{lem}
  For all primitive Mukai vectors $v \in N(X)$ the set of $v$-general stability conditions 
  is dense and open in $Stab^\dagger(X)$.
\end{lem}
\begin{proof}
  Choose an open subset $B^\circ$ with compact closure $B$.
  In the proof of \cite[Prop. 9.3.]{BridgelandK3} it is shown that 
  \[S=\{ E \in \cd^b(X) \,|\, E \text{ $\sigma$-semi-stable for some $\sigma \in B$}, v(E)=v \}\]
  has bounded mass. Hence \cite[Prop.\ 9.3]{BridgelandK3} applies and
  we get a wall an chamber structure on $B$ such that all objects $E
  \in S$ are stable outside a locally finite collection of walls.
  This shows the density. 
  The openness follows from  \cite[Prop.\ 9.4]{BridgelandK3} applied to $S$.
\end{proof}

\subsection{Reconstruction theorem} \label{sec:Reconstruction}

\begin{thm}\label{reconstruction_thm}
  Let $v \in N(X)$ be a standard vector and $\sigma \in Stab^\dagger(X)$ a $v$-general stability condition.
  \begin{enumerate}
  \item There exists a K3 surface $Y$ and an isomorphism of functors
    \[ \km^{\sigma}_X(v)  \isom \underline{Y} \]
    where $\underline{Y}$ is the functor $(Shm/\IC)^{op} \ra (set): T
    \mapsto Mor(T,Y)$.
  \item The Hodge structure $H^2(Y,\IZ)$ is isomorphic to the subquotient 
    of $\tilde{H}(X,\IZ)$ given by $v^\perp/ v$.
  \item The universal family $E \in \km^{\sigma}_X(v)(Y) \subset \cd^b(X \times Y)$ induces a derived 
    equivalence $\cd^b(X) \ra \cd^b(Y)$.
  \end{enumerate}
\end{thm}

\begin{proof}
  The proof consists of three steps.  First, we treat the case $v=v_0,
  \sigma \in U(X)$ and show that $\km_\sigma(v_0) \isom \underline{X}$ using
  Proposition \ref{v_0_semi_stable}.  Next we generalize to $v=v_0$
  and $\sigma \in Stab^\dagger(X)$ using Lemma \ref{cover_stab}.
  Finally the general case can be reduced to $v=v_0$ using Theorem
  \ref{reduction_to_lvl}.

{\em Step 1: Assume that $v_0=(0,0,1) \in N(X)$, and $\sigma \in U(X)$, then
\[ \km^{\sigma}(v_0) \isom \underline{X}. \] } 
Indeed, the morphism $\underline{X} \ra \km^{\sigma}_X(v_0)$ is given by  
\[ f: T \ra X \quad \mapsto \quad \ko_{\Gamma_f} \in  \km^{\sigma}_X(v_0)(T) \]
where $\Gamma_f \subset X \times T$ is the graph of $f$. We have to show
this map is an isomorphism. 

Injectivity: If we have two morphisms $f,g: T \ra X$ with $\ko_{\Gamma_f} \sim \ko_{\Gamma_g}$ 
then we claim that $f=g$. Indeed, by assumption there is a quasi-isomorphism 
\[ \ko_{\Gamma_f} \isom \ko_{\Gamma_g}[k] \tensor \pr_2^* L \qtext{ in } \cd^b(X \times T)\] 
for some $k \in 2 \IZ, L \in Pic(T)$. 
As $\ko_{\Gamma_f}$ and $ \ko_{\Gamma_g} \tensor \pr_2^* L$ are sheaves, we have $k=0$ 
and the quasi-isomorphism is an isomorphism of coherent sheaves. Moreover,
\[ L=\pr_{2 *}( \ko_{\Gamma_g} \tensor \pr_2^* L) \isom \pr_{2 *}(\ko_{\Gamma_f})=\ko_T. \]
Hence it is $\ko_{\Gamma_f}\isom \ko_{\Gamma_g}$ and it follows that $f=g$. 

Surjectivity: If $[E] \in \km^{\sigma}_X(v_0)(T)$, then $\DL i_t^* E$ is $\sigma$-stable of 
Mukai vector $v_0$. It follows from Proposition \ref{v_0_semi_stable} that $\DL i_t^* E\isom \ko_x[k]$ 
for a point $x \in X$ and $k \in 2\IZ$ depending on $t$.
Using the semi-continuity of $t \mapsto h^i(X,\DL i_t^* E)$ (Proposition  \ref{semicontinuity}) it is 
easy to see that the shift $k \in 2\IZ$ is independent of $t$ in each connected 
component $T_0 \subset T$.
Now we argue as in  \cite[Cor.\ 5.23.]{HuybrechtsFM} to see that there is a
morphism $f:T_0 \ra X$ line bundle $L \in Pic(T_0)$ such that 
$E|_{T_0} \isom \ko_{\Gamma_f} \tensor  \pr_2^* L$. 
Therefore $E \sim \ko_{\Gamma_f}$ for some $f: T \ra X$.
 
{\em Step 2. Assume that  $\sigma \in Stab^\dagger(X)$ is a $v_0$-general stability condition. 
Then $\km^{\sigma}(v_0)$ is isomorphic to $X$.\\ }
Indeed, by Lemma \ref{cover_stab} we find a $\Phi \in \WT(X)$ 
such that $\sigma'=\Phi_*( \sigma ) \in \overline{U}(X)$. Also note that we have
\[ T_{\ko_C(k)}^H(v_0) = v_0, \quad {T^2_A}^H(v_0)=v_0. \]
It follows that $v_0$-general stability conditions are mapped to $v_0$-general stability conditions.
By Lemma \ref{v_0_general} we conclude that $\sigma' \in U(X)$ and not in $\del U(X)$.

As we have seen in Theorem \ref{moduli_under_fm} the Fourier--Mukai 
equivalence $\Phi^{-1}$ induces isomorphisms of functors
\begin{align*}
  \km^{\sigma}(v_0) \isom \km^{\Phi_*\sigma}(\Phi^{H}_*v_0))  = \km^{\sigma'}(v_0) \isom \underline{X}.
\end{align*}
The last isomorphism is provided by step 1.

{\em Step 3. General case.}
Let $v$ be a standard vector and $\sigma \in Stab^\dagger(X)$ be a $v$-general stability condition.
By Theorem \ref{reduction_to_lvl} there is a K3 surface $Y$ with 
Hodge structure $H^2(Y,\IZ) \isom v^\perp/v$ and a derived 
equivalence $\Phi: \cd^b(X) \ra \cd^b(Y)$ respecting the distinguished component 
and mapping $v \in N(X)$ to $v_0 \in N(Y)$.

Since $\sigma \in Stab^\dagger(X)$ is $v$-general also $\Phi_*(\sigma) \in Stab^\dagger(Y)$ is $v_0$-general.  
By Theorem \ref{moduli_under_fm} the Fourier--Mukai transformation $\Phi: \cd^b(X) \ra \cd^b(Y)$ 
induces an isomorphism of moduli functors
$\km^{\sigma}_X(v)  \isom \km^{\Phi_* \sigma}_Y(v_0).$
Now we apply step 2 to conclude that $\km^{\Phi_* \sigma}_Y(v_0)\isom \underline{Y}$.

It remains to show that the universal family induces a derived equivalence.
This follows from the fact that, under the isomorphism  $\underline{Y} \isom \km^{\sigma}_X(v)$
the element $id_Y \in \underline{Y}(Y)$ maps to the kernel of the composition of 
the Fourier--Mukai equivalences used in the various reduction steps.
\end{proof}

\begin{rem}
  In general one expects that the moduli space $\km^{\sigma}(v)$ undergoes (birational)
  transformations, called wall-crossings
  when $\sigma$ moves in $Stab^\dagger(X)$.
  This behavior can be observed in our situation, too, but the transformations
  turn out to be isomorphisms.

  If $\sigma \in U(X)$, then $\km^\sigma(v_0)$ parametrizes the skyscraper sheaves $\ko_x, x \in X$.
  When $\sigma$ passes over a wall of type $(C_k)$, then the sheaves $\ko_x, x \in C$
  are replaced by the complexes $T_{\ko_C(k)} \ko_x$, whereas the sheaves $\ko_x, x \notin C$ 
  remain stable. 

  If $\sigma$ moves over an $(A)^{\pm}$-type wall, then all sheaves $\ko_x$ are
  replaced by the spherical twists $T_A^{\pm 2} \ko_x$. 
\end{rem}


\section{Appendix: Equivalences respecting $Stab^\dagger(X)$}
\label{sec:aut_dagger}
Let $\Phi: \cd^b(X) \ra \cd^b(Y)$ be a derived equivalence between two
K3 surfaces.  Recall from section \ref{sec:Stability}, that $\Phi$
{\em respects the distinguished component} if $\Phi_* Stab^\dagger(X)=Stab^\dagger(Y)$.

As we will see, this property can be verified for most of the known equivalences.
It is expected that $Stab(X)$ is connected and therefore it should in fact hold always.

We will use the following criterion, which is an easy consequence of
\cite[Prop.\ 10.3]{BridgelandK3} cf.\ Proposition \ref{stab_criterion}.
\begin{cor}\label{stab_eq_criterion}
  Let $\Phi: \cd^b(X) \ra \cd^b(Y)$ be a derived equivalence between two K3 surfaces.
  If the objects $\Phi(\ko_x), x \in X$ are $\sigma$-stable of the same phase
  for some $\sigma \in Stab^\dagger(Y)$, then $\Phi$ preserves the 
  distinguished component.
\end{cor}

As a direct application we find the following lemma.
\begin{lem}\label{easy_respect}
  The equivalences of derived categories listed below respect the distinguished 
  component of the stability manifold.
  \begin{itemize}
  \item {Shifts:} $[1]: A \mapsto A[1]$
  \item {Isomorphisms:} For $f: X \isom Y$, the functor $f_*: A \mapsto f_* A$
  \item {Line bundle twists:} For $L \in Pic(X)$, the functor $A \mapsto L \tensor A$
  \end{itemize}
\end{lem}

Before we can deal with more interesting auto-equivalences, we need a digression
on stability conditions.

\subsection{Large volume limit}\label{sec:limit_stable_objects}
Following \cite[Prop.\ 14.2]{BridgelandK3} we will show that families
of Gieseker-stable sheaves give rise to families of $\sigma$-stable
objects in stability conditions $\sigma$ near the large volume limit.

Let $M$ be a quasi-compact scheme over $\IC$. Denote by $i_m: X \ra M \times X$ the inclusion of 
the fiber over $m \in M(\IC)$.
For a sheaf $E \in Coh(X \times M)$ denote by $E_m$ the restriction $i_m^* E \in Coh(X)$
to the fiber over $m$.
\begin{prop}\label{StableFamilies}
  Let $h \in NS(X)$ be an ample class. 
  Let $\curly{E} \in Coh(M \times X)$ be an $M$-flat family of Gieseker-stable sheaves of fixed Mukai 
  vector $v(E_m)=v \in N(X)$. Assume that $r(E_m)>0$ and $\mu(E_m)=\mu_h(E_m) > 0$.

  Then there exists a $n_0 \geq 1$ such that
  the objects $E_m, m\in M(\IC)$ are stable with respect to the stability condition 
  $\sigma(0, n h ) \in V(X)$ for all $n \geq  n_0$. 
\end{prop}
\begin{proof}
  We will go through Bridgeland's arguments for the case of a
  single Gieseker-semi-stable sheaf and verify that they suffice to cover our
  situation.  Our presentation fills in some details which were not
  explicitly mentioned in \cite{BridgelandK3}.

  We first note that the heart $\ka(0,n h)$ is independent of $n$
  and the objects $E=E_m$ lie in the heart $\ka(0,n h)$.

  Suppose $0 \neq A \ra E$ is a proper sub-object of $E$  in $\ka(0,h)$.
  We have the following formula
  \begin{align}\label{Dformula} \frac{Z_n(E)}{r(E)} - \frac{Z_n(A)}{r(A)} 
    = - (\nu(E)-\nu(A)) + i n (\mu(E)-\mu(A)) =: \Delta_n,
  \end{align}
  where $\nu(A)=s(A)/r(A)$.

  The inequalities $\arg(Z_n (A)) < \arg( Z_n (E))$ and
  $\arg(Z_n(E)) < \arg(\Delta_n)$ are equivalent\footnote{
  We use the convention that, for $z \in \IC^{*}$ the argument $arg(z)$
  is the unique real number in $[0,2)$ such that $z=r\, exp(i \pi\, arg(z))$ for
  some $r \in \IR_{>0}$.}.
  Note that we have $Z_n(E) \in \IH$ by the assumption $\mu(E)>0$, if also $\Delta_n \in \IH$,
  then $\arg(Z_n(E)) < \arg(\Delta_n)$ is equivalent to
  \[ \frac{Re(Z_n(E))}{Im(Z_n(E))} > \frac{Re (\Delta_n)}{Im (\Delta_n)} = 
  -\frac{\nu(E) - \nu(A) }{n(\mu(E)-\mu(A))}.  \tag{\#} \]
  
  We claim that $\Delta_n \in \IH$ unless $A$ is a subsheaf
  and $\mu(A)=\mu(E)$, in which case $\Delta_n \in \IR_{<0}$.
  
  Indeed, consider the exact sequence of cohomology sheaves:
  \begin{align}\label{DAEB} 0 \lra D \lra A \lra E \lra B \lra 0 \end{align}
  where $D = \mathcal{H}^{-1}( Cone(A \ra E) )$ and 
  $B = \mathcal{H}^{0}( Cone(A \ra E) )$.
  Note that, $A=\mathcal{H}^{0}(A)$ since $\mathcal{H}^{-1}(A)=\mathcal{H}^{-1}(E)=0$.
  Let $E' \subset E$ be the image of $A \ra E$. We get short exact sequences
  \begin{align}\label{ADEEB} 0 \lra D \lra A \lra E' \lra 0 \qtext{and} 0 \lra E' \lra E \lra B \lra 0.
  \end{align}
  If $D=0$, then $A \ra E$ is a subsheaf and $Z_n(A) \in \IH$ unless $\mu(A)=\mu(E)$.
  If $D \neq 0$, then we have $\mu(D) \leq 0  < \mu(A)$ and therefore also $\mu(A) < \mu(E')$. 
  Hence $\mu(A) < \mu(E') \leq \mu(E)$ by stability of $E$. This shows the claim.

  In the case $\Delta_n \in \IR_{<0}$ we always have $\arg(Z_n E) < \arg(\Delta_n)=1$.
  Thus we may exclude this case henceforth.
  
  Now, the quotient $Re(Z_n(E))/Im(Z_n(E))$ converges to $+\infty$ for $n \ra \infty$.
  Hence it suffices to bound $-(\nu(E) - \nu(A))/n(\mu(E)-\mu(A))$ form above.
  The numerator can be bounded using the following lemma proved below.
  \begin{lem}
    The set \[ \{ \nu(A) \,|\, A \ra E_m \; \text{sub-object in}\; \ka(0,h), \; m \in M(\IC) \} \]  
    is bounded from above. 
  \end{lem}
  It remains to find a constant $C$ such that $\mu(E)-\mu(A)>C>0$. \\
  Case $\mu(E') < \mu(E)$: As $r(E') \leq r(E)$ we have $\mu(E)-\mu(E')> 1 / r(E)^2$
  and since $\mu(A) \leq \mu(E')$ the same bound holds for $\mu(E)-\mu(A)$ as well. \\
  Case $\mu(E') = \mu(E)$: 
  If $D=0$, then $A \ra E$ is a subsheaf and again $\mu(E)-\mu(A) > 1 / r(E)^2$.
  If $D\neq 0$, then the exact sequence (\ref{ADEEB}) and  $\mu(D)\leq 0$ shows that
  \begin{align*}
    \mu(A)= \mu(D) \frac{r(D)}{r(E')+r(D)} + \mu(E') \frac{r(E')}{r(E')+r(D)} \leq \mu(E) \frac{r(E)}{r(E)+1}.
  \end{align*}
  As $r(E)=r(E_m)$ is independent of $m$ we get a uniform bound.
\end{proof}

\begin{proof}[Proof of Lemma]
  Recall that if $v(A)=(r,l,s)$, then $\nu(A)=s/r$. 
  The Euler-characteristic of $A$ is computed as
  \[ h^0(A)-h^1(A)+h^2(A)=\chi(A)=\chi(\ko_X,A)=-(1,0,1).(r,l,s)=r+s  \]
  hence it suffices to bound $\chi(A)/r$ from above. 
  As $\mu^{min}(A) > 0 = \mu(\ko_X)$ we have $Hom(A,\ko_X)=H^2(A)=0$.
  Therefore $\chi(A)/r \leq h^0(A)/r$.
  The long exact sequence
  \[ 0 \lra H^0(D) \lra H^0(A) \lra H^0(E') \lra H^1(D) \lra \dots  \]
  shows that $h^0(A) \leq h^0(D) + h^0(E')$.
  Moreover, $h^0(E') \leq h^0(E)$ and $h^0(E)=h^0(E_m)$ is bounded uniformly in
  $m \in M(\IC)$, since $h^0(E_m)$ is semi-continuous and $M$ quasi-compact.
  In the case $D=0$ we are done. Let now $D \neq 0$.
  Note that $r=r(A) \geq r(D)$, and hence $h^0(D)/r(A) \leq h^0(D)/r(D)$.
  Therefore it suffices to bound 
  \[ \Set{h^0(D)/r(D)}{ D = \curly{H}^{-1}( Cone(A \ra E_m)), m \in M(\IC)  }. \]
  {\em Claim: $h^0(D)/r(D) \leq 1$ for all $D \neq 0$ torsion
    free with $\mu^{min}(D) \leq 0$.}

  Let $0 \neq s \in H^0(D)$ be a section. As in \cite[Lem.\ 14.3]{BridgelandK3} 
  we show that $s: \ko_X \ra D$ is injective and that $Q=coker(s)$
  is torsion free. Note that also $\mu^{min}(Q) \geq \mu^{min}(D) \geq 0$ unless $Q=0$.
  Hence $Q$ fulfills the same assumptions as $D$ and we can apply induction on $r(D)$.\\
  If $r(D)=1$, then $h^0(D)=0$ as $D$ has no non-trivial
  torsion-free quotients. \\
  If $r(D)> 1$, then $h^0(D)=h^0(Q)+1$, since $h^1(X,\ko_X)=0$, and therefore
  \[ \frac{h^0(D)}{r(D)} = \frac{h^0(Q)+1}{r(Q)+1} \leq 1 \]
  by induction hypothesis.
\end{proof}

\subsection{Moduli spaces and spherical twists}
A very important class of derived equivalences between K3 surfaces is provided by
moduli spaces of sheaves. 

\begin{prop}\label{sheaf_respect} 
  Let $M=M_h(v)$ be a fine, compact, two-dimensional moduli-space of
  Gieseker-stable sheaves on $X$ and $\Phi: \cd^b(M) \sra \cd^b(X)$ the
  Fourier--Mukai equivalence induced by the universal family (cf. \cite[Sec.\ 10.2]{HuybrechtsFM}). \\
  Then $\Phi$ respects the distinguished component.
\end{prop}

\begin{proof}
  Let $\ko(1)$ be the ample line bundle with $c_1(\ko(1))=h$.
  Tensoring with a large power of $\ko(1)$ we reduce to the case $\mu(E)>0$.
  Now Proposition \ref{StableFamilies} applies and there is an $n>0$ 
  such that all the sheaves $E_m, m\in M(\IC)$ are stable
  in the stability condition $\sigma(0,n h)\in U(X)$.
  As all sheaves $E_m$ lie in the heart $\ka(0,h)$ they have the same phase.
  Hence the proposition follows from Corollary \ref{stab_eq_criterion}.
\end{proof}

\begin{prop}\label{spherical_respect}
  Let $A$ be a spherical vector bundle, which is Gieseker-stable with
  respect to  an ample class $h \in NS(X)$. \\
  Then the spherical twist $T_A: \cd^b(X) \ra \cd^b(X)$ respects the
  distinguished component.
\end{prop}
\begin{proof}
  The spherical twist functor has Fourier--Mukai kernel
  \[ \kp = Cone(pr_1^* A\dual \tensor pr_2^* A \xlra{tr} \ko_{\Delta}) \in \cd^b(X \times X) \]
  cf.\ \cite[Def.\ 8.3]{HuybrechtsFM}.
  Let $i_x: X \ra \{x\} \times X \subset X \times X$ be the inclusion of the fiber.
  Since $i_x^*(\kp)[1]=T_A(\ko_x)[1]=:E_x$ is a sheaf, the complex $\kp[1]$
  is quasi isomorphic to a sheaf which is flat along $\pr_1$.

  Mukai shows in \cite[Rem.\ 3.11.]{Mukai1987} that the sheaves $E_x$ are Gieseker-stable.
  Therefore $\kp[1]$ induces a map $f:X \ra M_h(v_1)$, where $v_1=v(T_A(\ko_x))$,
  which is easily seen to be an isomorphism. 
  Hence $T_A$ is a special case of a Fourier--Mukai
  transformation at a moduli space of Gieseker-stable sheaves and Proposition \ref{sheaf_respect} applies.
\end{proof}

Using Bridgeland's description the boundary of $U(X)$ in \cite[Thm.\ 12.1]{BridgelandK3} (cf.\ Theorem \ref{U_boundary}) we can also show that spherical twists along torsion sheaves do respect 
the distinguished component.

\begin{prop}\label{boundary_lemma}\label{curve_respect}
  Let $C$ be a $(-2)$-curve on a K3 surface $X$ and $k \in \IZ$, then
  the spherical twist $T_{\ko_C(k)}$ does respect the distinguished component.
\end{prop}
\begin{proof}
  We will show that every pair $(C,k)$ does define a non-empty boundary component of 
  $U(X)$ of type $(C_k)$. Then \cite[Thm.\ 12.1]{BridgelandK3} shows that 
  ${T_{\ko_C(k)}}_* \overline{U}(X) \cap \overline{U}(X) \neq \emptyset$ 
  and therefore $T_{\ko_C(k)}$ respects the distinguished component, cf.\ Remark \ref{WT_Dagger}.

  Every $(-2)$-curve $C$ defines a boundary component of the ample cone,
  i.e. there is a class $\eta \in \overline{Amp(X)}$ such that $C.\eta = 0$ 
  and $C'.\eta > 0$ for all other $(-2)$-curves $C'$.
  Multiplying with a positive number we can assume that $\eta^2>2$.

  We claim that there is always a  $\beta \in NS(X)_\IR$ such that 
  \begin{enumerate}
  \item $exp(i \eta + \beta).\delta \neq 0$ for all $\delta \in \Delta(X)$, 
    i.e. $exp(i \eta + \beta) \in \kp_0^+(X)$.
  \item $exp(i \eta + \beta).\delta \notin \IR_{\leq 0}$ for all $\delta \in \Delta^{>0}(X)$ and
  \item $\beta.C + k \in (-1,0)$.
  \end{enumerate}     
  Indeed, for $\delta=(r,l,s)$ we have
  \[ Im(exp(i \eta + \beta).\delta)=l.\eta - r \beta.\eta. \]
  This number is non-zero if $r \neq 0$ and $\beta.\eta \neq l.\eta/r$.
  If $r=0$, then $\delta^2=l^2=-2$ and $exp(i \eta + \beta).\delta=l.\eta=0$ implies that
  $l = \pm C$ by construction of $\eta$.
  In this case $Re(exp(i \eta + \beta).\delta)=\pm \beta.C - s$ is nonzero if $(3)$ is 
  fulfilled. Thus it suffices to chose $\beta$ in such a way that the countably many inequalities
  $\beta.\eta \neq l.\eta/r, \, l\in NS(X)$ and the open condition $(3)$ hold. 
  This shows the claim.

  Let $\sigma$ be the unique stability 
  condition in $\overline{U}(X)$ with central charge $\exp(i \eta + \beta)$.
  Note that $\sigma$ does not lie on a boundary component of type $(A^{\pm})$ by $(2)$.
  By construction, if $x \in C$ then $\ko_x$ is destabilized by a sequence
  \[ 0 \lra \ko_C(n+1) \lra \ko_x \lra \ko_C(n)[1] \lra 0. \]
  Thus  $\sigma$ is a general point of a boundary component of $U(X)$ of type $(C_n)$ 
  for some $n$. The number $n$ is uniquely determined by the property that
  \[ Z_{\sigma}(\ko_C(n)[1]), \; Z_{\sigma}(\ko_C(n+1)) \in \IH \cup \IR_{<0}\]
  which has to hold since $\ko_C(n+1)$ and $\ko_C(n)[1]$ lie in the heart of $\sigma$.
  This is equivalent to $-1 \leq \beta.C + n < 0$. Hence $k=n$ by condition $(3)$.
\end{proof}

\begin{rem}
  The general question if for a spherical object $A \in \cd^b(X)$ the equivalence
  $T_A$ respects the distinguished component remains open -- even in the case that 
  $A$ is a vector bundle. 
\end{rem}


\subsection{Auto-equivalences and the K\"ahler moduli space}

It was shown by \cite{HLOY2004},\cite{PloogPHD} and \cite{HuybrechtsMacriStellariOrientation2009}
(cf.\  Theorem \ref{orientation})  that the image of the map
\[ Aut(\cd^b(X)) \lra O_{Hodge}(\HT(X,\IZ)) \]
is the index two subgroup $O^+_{Hodge}(\HT(X,\IZ))$. 

\begin{prop}\label{surj_of_aut}
  Let $Aut^\dagger(\cd^b(X)) \subset Aut(\cd^b(X))$ be the subgroup of
  auto-equivalences which respect the distinguished component. 
  Then 
  \[ Aut^\dagger(\cd^b(X)) \lra O_{Hodge}^+(\HT(X,\IZ)) \]  
  is surjective.
\end{prop}

\begin{proof}
  As explained in \cite[Cor.\ 10.13.]{HuybrechtsFM} every element of $O_{Hodge}^+(\HT(X,\IZ))$
  is induced by the composition of derived equivalences of the following type.
  \begin{enumerate}
    \item Line bundle twists: For $L \in Pic(X)$, the functor $L \tensor \_ \in Aut(\cd^b(X))$.
    \item For isomorphisms $f: X \ra Y$, the functor $f_*: \cd^b(X) \ra \cd^b(Y)$.
    \item For fine, compact, two-dimensional moduli spaces $M$ of Gieseker-stable sheaves with 
      universal family $\ke$, the Fourier--Mukai transform
      \[ FM(\ke): \cd^b(M) \ra \cd^b(X). \]
    \item Spherical twists along $\ko_X$.
    \item Spherical twists along $\ko_C$ for a $(-2)$-curve $C \subset X$.
  \end{enumerate}
  All these equivalences do respect the distinguished component due to our
  Lemma \ref{easy_respect} for (1),(2), Proposition \ref{sheaf_respect} for (3), Proposition \ref{spherical_respect} for (4) and
  Proposition \ref{curve_respect} for (5).
\end{proof}

This result enables us prove the alternative description of the K\"ahler moduli space 
using the stability manifold, alluded to in Remark \ref{KM_stab_constr}.
We use the notation from Section 2 and 3.

\begin{cor}\label{double_quot}
  We have
  \[    Aut^\dagger(\cd^b(X)) \setminus Stab^\dagger(X) / \GLT \isom KM_0(X) \]
  where $KM_0(X)=\Gamma_X \setminus \pd_0(N(X))  \subset KM(X)$.
\end{cor}
\begin{proof}
  Recall that $Aut_0^\dagger(\cd^b(X))$ is the subgroup of $Aut^\dagger(\cd^b(X))$ of 
  auto-equivalences acting trivially on $\HT(X,\IZ)$.
  By \cite[Thm.\ 1.1]{BridgelandK3} (cf.\ Theorem \ref{stability_theorem}) the quotient 
  $Aut^\dagger_0(\cd^b(X)) \setminus Stab^\dagger(X)$ is identified with the period 
  domain $\kp_0^+(X) \subset N(X)_\IC$ via $\pi: Stab^\dagger(X) \ra \kp_0^+(X)$. 
  As $\pi$ is $\GLT$-equivariant, we have 
  \[ Aut^\dagger_0(\cd^b(X)) \setminus Stab^\dagger(X) / \GLT \isom \kp_0^+(X)/\GL \isom \pd_0^+(X). \]
  Now Proposition \ref{surj_of_aut} shows that 
  \[ Aut^\dagger(\cd^b(X)) \setminus \pd_0^+(X) \isom O_{Hodge}^+(\HT(X,\IZ))\setminus \pd_0^+(X) \isom KM_0(X).
  \qedhere \]
\end{proof}


\subsection{Reduction to the large volume limit}\label{sec:reduction_to_lvl}

As another consequence we obtain the following proposition which allows us to reduce 
many statements about objects with a standard Mukai vector $v$ (cf.\ Definition \ref{def_standard_vector})
to the special case $v=v_0=(0,0,1) \in N(X)$, which is the class of a point sheaf.

\begin{prop}\label{reduction_to_lvl}
  Let $v \in N(X)$ be a standard vector. 
  Then there is a K3 surface $Y$ and a derived equivalence $\Phi:\cd^b(X) \ra \cd^b(Y)$ such that
  \[  \Phi^H(v) = v_0 \]
  and $\Phi$ respects the distinguished component.

  Moreover, $Y$ is a fine moduli space of Gieseker-stable sheaves and the Hodge 
  structure $H^2(Y,\IZ)$ isomorphic to the subquotient $v^\perp/v$ of $\HT(X,\IZ)$.
\end{prop}

\begin{proof}
  Write $v=(r,l,s)$. 
  Applying $T_{\ko_X}$, shifts and $L \tensor \_ $ for a line bundle $L$ as in the 
  proof of \cite[10.10]{HuybrechtsFM} we reduce to the case that that $r>0$.
  Note that the equivalences $T_{\ko_X}$, $[1]$, $L \tensor \_ $ respect the 
  distinguished component by Lemma \ref{easy_respect} and Proposition \ref{spherical_respect}.

  By \cite[Sec.\ 10.2]{HuybrechtsFM} there is an ample class $h \in
  NS(X)$ such that the moduli space of Gieseker-stable sheaves $Y=M_h(v)$ is a K3 surface
  with Hodge structure $H^2(Y,\IZ) \isom v^\perp/v$ as subquotient of $\HT(X,\IZ)$.
  The derived equivalence induced by the universal bundle $\ke$
  respects the distinguished component by Proposition \ref{sheaf_respect}
  and maps $v_0$ to $v$.
\end{proof}


\bibliographystyle{alpha}
\bibliography{References}

\end{document}